\documentclass[pdflatex,sn-mathphys-num]{sn-jnl}

\usepackage{graphicx}%
\usepackage{multirow}%
\usepackage{amsmath,amssymb,amsfonts}%
\usepackage{amsthm}%
\usepackage{mathrsfs}%
\usepackage[title]{appendix}%
\usepackage{xcolor}%
\usepackage{textcomp}%
\usepackage{manyfoot}%
\usepackage{booktabs}%
\usepackage{algorithm}%
\usepackage{algorithmicx}%
\usepackage{algpseudocode}%
\usepackage{listings}%

\makeatletter
\@ifpackageloaded{geometry}%
  {\geometry{a4paper,margin=1in}}%
  {\usepackage[a4paper,margin=1in]{geometry}}
\makeatother
\usepackage{amsmath,amssymb,amsthm,mathtools}
\usepackage{booktabs}
\usepackage{subcaption}
\usepackage{microtype}
\usepackage[numbers,sort&compress]{natbib}
\makeatletter
\@ifpackageloaded{hyperref}%
  {\hypersetup{colorlinks=true,linkcolor=blue,citecolor=blue,urlcolor=blue,pdftitle={Exact Finite-Horizon Memory, Conditioning, and Dissipative Decay in Coarse Upwind Finite-Volume Prediction},pdfauthor={Antonis Polemitis, Nicholas Christakis, Dimitris Drikakis},}}%
  {\usepackage[colorlinks=true,linkcolor=blue,citecolor=blue,urlcolor=blue]{hyperref}\hypersetup{pdftitle={Exact Finite-Horizon Memory, Conditioning, and Dissipative Decay in Coarse Upwind Finite-Volume Prediction},pdfauthor={Antonis Polemitis, Nicholas Christakis, Dimitris Drikakis},}}
\makeatother
\makeatletter\let\orcidlogo\@undefined\makeatother
\usepackage{orcidlink}
\usepackage[nameinlink]{cleveref}

\newtheorem{theorem}{Theorem}
\newtheorem{lemma}{Lemma}
\newtheorem{proposition}{Proposition}
\newtheorem{corollary}{Corollary}
\newtheorem{remark}{Remark}

\newcommand{\R}{\mathbb R}
\newcommand{\epsmach}{\epsilon_{\rm mach}}
\DeclareMathOperator{\rank}{rank}

\DeclareMathOperator{\TV}{TV}
\DeclareMathOperator{\card}{card}

\theoremstyle{thmstyleone}%

\theoremstyle{thmstyletwo}%

\theoremstyle{thmstylethree}%

\begin{document}

\title[Article Title]{Exact Finite-Horizon Memory, Conditioning, and Dissipative Decay in Coarse Upwind Finite-Volume Prediction}


\author[1]{\fnm{Antonis} \sur{Polemitis}\,\orcidlink{0000-0003-3188-8746}}\email{polemitis.an@unic.ac.cy}

\author*[2]{\fnm{Nicholas} \sur{Christakis}\,\orcidlink{0000-0001-9212-2336}}\email{christakis.n@unic.ac.cy}

\author[2]{\fnm{Dimitris} \sur{Drikakis}\,\orcidlink{0000-0002-3300-7669}}\email{drikakis.d@unic.ac.cy}

\affil[1]{\orgdiv{UNIC Evolve}, \orgname{University of Nicosia}, \city{Nicosia}, \postcode{CY-2417},  \country{Cyprus}}

\affil[2]{\orgdiv{Institute for Advanced Modelling and Simulation}, \orgname{University of Nicosia}, \city{Nicosia}, \postcode{CY-2417},  \country{Cyprus}}


\abstract{Coarse finite-volume averages do not generally form a predictive state:
realized interface fluxes close a conservative update, but distinct fine-grid states with identical parent averages can generate different future coarse histories. We analyze this failure for periodic scalar
advection discretized by a first-order upwind finite-volume method with forward Euler time integration. We derive a finite-horizon observation-rank law: each additional observation exposes one new child-cell layer and contributes one fewer independent direction than
the number of parent cells, until the unresolved layers are exhausted. Centralized prediction therefore requires one fewer additional
coordinate than the number of parent cells per exposed layer, whereas product-local prediction requires one coordinate per layer in each parent. An anchored flux-divergence queue attains the centralized bound,
identifies the periodic flux gauge, and, at saturation, forms a minimal autonomous predictive state with the parent averages. We then distinguish exact observability from stable recoverability. The collar-to-queue map becomes rapidly ill-conditioned as the Courant number decreases, so algebraically visible delayed information may fall below a prescribed
numerical tolerance. For Courant numbers strictly between zero and one, we prove contraction of the nonconstant component under bounded arithmetic perturbations, with separate control of mean drift, an
explicit perturbation neighborhood, and a grid-dependent decay time. Numerical experiments illustrate the rank ladder, effective-rank loss, queue conditioning, step-function flattening, and delayed coarse separation followed by dissipative decay. The results provide a solvable benchmark for assessing state sufficiency in coarse, reduced, multiscale, and learned scientific models.}




\maketitle

\color{black}
\section{Introduction} \label{intro}

Finite-volume methods evolve cell averages through numerical fluxes at cell
interfaces \cite{LeVeque:2002fvm,EymardGallouetHerbin:2000fvm,
MoukalledManganiDarwish:2016fvm}.  This conservative structure is one reason
why finite-volume schemes remain central in computational fluid dynamics (CFD),
shock capturing, adaptive mesh refinement, and hyperbolic conservation laws
\cite{CrandallMajda:1980monotone,Berger:1987interfaces,
BergerColella:1989amr,KurganovTadmor:2000central,
GottliebShuTadmor:2001ssp}.  A related question arises whenever a fine
calculation is restricted, sampled, reconstructed, or replaced on a coarser
grid: do the coarse cell averages alone contain enough information to
predict their own future evolution?

In general, the answer is no.  A conservative restriction fixes the amount
of material in a parent cell but not where that material lies inside the
parent.  If a parent cell consists of two fine cells, then the fine states
\[
  (1,0)
  \qquad\hbox{and}\qquad
  (0,1)
\]
have the same parent average.  For right-moving upwind advection, however,
the second state has material closer to the right parent interface and can
affect the neighboring parent sooner.  The parent average is therefore not an
exact predictive state.

This observation is simple, but its consequences are easy to obscure.  In
coarse/fine coupling, a completed fine step can be made conservative on the
coarse grid by communicating realized interface fluxes.  This is the
mechanism behind refluxing and related multilevel corrections
\cite{OsherSanders:1983localtime,ConstantinescuSandu:2007multirate,Thornber:2008entropy}.  Such a
realized flux is a one-step message.  It does not generally provide all
future fluxes, unless additional subcell information or a predictive memory
state is retained.

The same distinction appears in reduced modeling and large-eddy simulation.
Filtering or averaging produces unresolved stresses, fluxes, or memory
terms.  Mori--Zwanzig \cite{mori1965,zwanzig1961} and optimal-prediction viewpoints make this explicit:
eliminating variables generally produces non-Markovian resolved dynamics
\cite{ChorinHaldKupferman:2002memory,ParishDuraisamy:2017nonmarkovian,
GouasmiParishDuraisamy:2017memory,LinTianLivescuAnghel:2021mzlearning}.
Recent finite-volume closure work also emphasizes that exact unresolved
terms can be written down for a declared discretization, even if producing or
approximating them remains difficult \cite{Agdestein:2026exactfvclosure}.

The issue has become more visible as machine learning methods have been used
for flow super-resolution, spatiotemporal forecasting, and data-driven
finite-volume or flux-form prediction
\cite{DuraisamyIaccarinoXiao:2019data,KarniadakisEtAl:2021piml,
VinuesaBrunton:2022cfdml,BarSinai:2019datadriven,KochkovEtAl:2021mlcfd,
ZhuangEtAl:2021learnedadvection,deRomemontEtAl:2026learnedfv,
Gao:2021physicsinformed,Gao:2023bayesian,
Sofos:2025spatiotemporal,
Liu:2026entropyCFN,Charles:2025dspweno}.  In such settings, a coarse input
may be compatible with many fine states.  Some hidden fine-scale differences
may matter for future fluxes, while others may decay because of physical or
numerical dissipation.  The distinction is essential: exact algebraic
information, stable recoverability, and practical finite-precision relevance
are different notions.

The present paper isolates these notions in a deliberately simple model:
periodic scalar advection with first-order positive upwind forward Euler.
The model is simple enough to allow exact theorems, but rich enough to show
the main mechanism.  Our contribution is not a broad impossibility theorem
for turbulence, Euler flows, or learned closures.  Instead, it is a precise
finite-volume calculation showing what is true in the scalar upwind
hierarchy, and why exact rank must be interpreted together with conditioning
and dissipation.

\subsection{Contributions}

Observability, minimal realization, refluxing, unresolved-variable memory,
and upwind numerical dissipation are individually classical.  The
contribution here is their exact finite-volume synthesis for a declared
coarse/fine hierarchy: the horizon-by-horizon dimension count,
locality-dependent lower bounds, attaining queue coordinates, periodic gauge
reduction, and quantitative separation of algebraic visibility from stable
and long-time relevance. The paper makes five contributions.

\begin{enumerate}
\item We establish the realized-flux identity as the classical one-step
baseline: once the integrated interface fluxes of a completed fine step are
known, parent averages update exactly by flux telescoping.  This explains why
refluxing works, but also why a completed flux is not automatically a
recursive predictive state.

\item We prove the exact finite-horizon observation-rank formula
\[
  \rank {\cal O}_L =
  P+(P-1)\min\{L,r-1\}
\]
for a periodic scalar upwind hierarchy with \(P\) parents and \(r\) children
per parent.  This yields continuous centralized and parent-local memory
lower bounds.

\item We construct attaining predictive coordinates: raw local queues attain
the parent-local bounds, while an anchored flux-divergence queue attains the
centralized bound.  At saturation, the resulting state is autonomous and has
minimal total dimension among continuous autonomous encoders on open support.

\item We quantify stable observability.  The queue map for delayed subcell
information is triangular with diagonal entries
\(\lambda,\lambda^2,\ldots,\lambda^q\).  Hence exact rank can overstate what
is recoverable in finite precision when \(\lambda^q\) approaches the
declared singular-value tolerance.

\item We prove dissipative decay under bounded arithmetic perturbations.  For
\(0<\lambda<1\), the nonconstant fine-grid component contracts toward an
explicit perturbation-controlled neighborhood of the computed mean, while
the computed mean satisfies a separate accumulated-drift estimate.  The
endpoint \(\lambda=1\) is nondissipative.
\end{enumerate}




\color{black}
\subsection{Relation to turbulence, multiscale physics, and scientific computing}

This issue is especially consequential in turbulence, where coarse representation is not
simply a reduction in mesh resolution. Filtering removes nonlinear scale interactions whose
influence re-enters the resolved equations through subgrid-scale stresses and, more generally,
through history-dependent closure terms. Classical large-eddy simulation (LES) made this
dependence explicit through eddy-viscosity closure and its dynamic refinement
\cite{smagorinsky1963,germano1991}, while projection-operator formulations show that
eliminating unresolved degrees of freedom generically produces memory and fluctuating
forces \cite{zwanzig1961,mori1965}. The present analysis does not claim a
Navier--Stokes turbulence closure, nor that its scalar rank formula transfers unchanged to
turbulent flow. Rather, it provides a fully solvable finite-volume benchmark in which three
requirements that are often conflated in turbulence capture are separated exactly: the
dimension of the information required for pathwise coarse prediction, the conditioning with
which that information can be reconstructed, and the time over which it remains dynamically
relevant under dissipation. It thereby provides a precise example of why conservation or
high-fidelity reconstruction of coarse observables need not determine the future interface
fluxes of the same realization, and why algebraically necessary subgrid information may
nevertheless be poorly recoverable or rapidly attenuated at a prescribed tolerance.

The same state-sufficiency question extends beyond classical CFD to scientific-computing
methods that advance reduced observables while delegating hidden degrees of freedom to a
fine-scale simulator, a reduced representation, or a learned closure. Equation-free and
heterogeneous multiscale methods use short microscopic simulations to supply or advance
macroscopic information \cite{kevrekidis2003,eengquist2003}; proper orthogonal
decomposition and dynamic mode decomposition seek low-dimensional coordinates for
coherent dynamics \cite{sirovich1987,schmid2010}; and sparse system identification and data-driven discretization infer governing dynamics or numerical operators from observations
\cite{brunton2016, BarSinai:2019datadriven}. In these settings, the exact rank and attaining
flux-divergence queue derived here provide a benchmark for state sufficiency, while the
singular-value and dissipative estimates impose two additional requirements: the retained
state must be stably recoverable and must persist over the intended rollout horizon. The
resulting perspective suggests a general validation protocol for reduced, multiscale, and
learned models of physical systems: specify the retained observables, prediction horizon,
locality or communication architecture, recovery tolerance, and dissipation mechanism before
asserting closure, Markovianity, or autonomous prediction.

\color{black}
Active Flux methods provide a complementary finite-volume perspective by
evolving interface point values together with cell averages
\cite{Barsukow:2019activeflux,Calhoun:2023activeflux}.  Their additional interface degrees of freedom are part of a designed high-order discretization.  The queue derived here has a different status: it is not proposed as a new discretization, but is the sharp finite-horizon information register for a fixed first-order upwind hierarchy.  Moreover, centralized
periodic prediction depends only on flux divergences and therefore removes
one spatially uniform interface-flux gauge at every queue time.

The scalar upwind problem below is therefore used as a clean reference
problem where exact closure dimension, conditioning, and dissipative loss can
all be computed.  The paper does not claim that the scalar rank formula transfers unchanged to turbulence, nonlinear conservation laws, or general learned closures.

This is useful precisely because it separates questions that are easy to
mix:
\begin{itemize}
\item exact pathwise closure: what information determines the future coarse
trajectory of the same fine realization?
\item stable numerical closure: can that information be recovered with
reasonable conditioning?
\item finite-precision relevance: does unresolved information remain above a declared numerical tolerance or perturbation floor after a dissipative numerical method has acted?
\end{itemize}

\paragraph{Organization.}
The rest of the paper is organized as follows.  \Cref{sec:fv_closure}
states the one-step finite-volume flux-closure identity and separates
completed flux communication from recursive prediction.  \Cref{sec:periodic}
defines the periodic scalar upwind hierarchy used in the analysis.
\Cref{sec:exact_fh} proves the exact finite-horizon memory theorem, including
centralized and parent-local encoder lower bounds and the flux-divergence
queue construction.  \Cref{sec:stable-observability} studies stable
observability and defines the effective floating-point rank used in the
experiments.  \Cref{sec:diss_eras} proves dissipative erasure for the
nonconstant component of the upwind solution.  \Cref{sec:num_exp} reports the
rank, singular-value, conditioning, step-flattening, delayed-collision, and
flux-closure experiments.  \Cref{sec:discussion} interprets the results and \Cref{sec:limitations} summarizes limitations and future work before the
conclusions in \Cref{sec:conc}.

\section{Finite-volume closure and one-step flux messages} \label{sec:fv_closure}

Let a one-dimensional fine grid be partitioned into contiguous parent cells.
Fine cell \(i\) has width \(h_i>0\), and parent \(K\) has total width
\[
  H_K=\sum_{i\subset K} h_i .
\]
The restriction \(R\) maps fine cell averages \(U_i\) to parent averages:
\[
  (RU)_K =
  \frac{1}{H_K}\sum_{i\subset K} h_i U_i .
\]
Many fine states have the same parent averages.  Exact coarse prediction is
therefore a factorization question: can the restricted fine update be written
as a function of the declared coarse information?

The first observation is that one completed conservative fine step is always
closed by the realized integrated fluxes through parent interfaces.

\begin{proposition}[Realized-flux one-step closure]
\label{prop:flux-closure}
Suppose one fine finite-volume step has conservative form
\[
  U_i^{n+1}
  =
  U_i^n
  -
  \frac{\Phi_{i+1/2}-\Phi_{i-1/2}}{h_i},
\]
where \(\Phi_{i+1/2}\) is the realized time-integrated numerical flux through
face \(i+1/2\).  Let \(Q_K\) denote the realized flux through parent interface
\(K+1/2\).  Then the parent update satisfies
\[
  (RU^{n+1})_K
  =
  (RU^n)_K
  -
  \frac{Q_K-Q_{K-1}}{H_K}.
\]
\end{proposition}

\begin{proof}
Multiply the fine update by \(h_i\) and sum over the children of parent
\(K\):
\[
  \sum_{i\subset K}h_i U_i^{n+1}
  =
  \sum_{i\subset K}h_i U_i^n
  -
  \sum_{i\subset K}(\Phi_{i+1/2}-\Phi_{i-1/2}).
\]
All internal fine faces cancel. Only the two exterior parent faces remain, giving
\[
  \sum_{i\subset K}h_i U_i^{n+1}
  =
  \sum_{i\subset K}h_i U_i^n
  -
  (Q_K-Q_{K-1}).
\]
Division by \(H_K\) proves the claim.
\end{proof}

This proposition explains the difference between a one-step conservative
message and a recursive predictive state.  The realized interface register \(Q_K\) can be
sufficient after the fine computation has already produced it.  It need not
determine future realized fluxes.  The rest of the paper studies the memory
needed for repeated prediction in the scalar upwind hierarchy.

\section{Periodic scalar upwind hierarchy} \label{sec:periodic}

The discrete hierarchy below corresponds to the constant-speed advection
equation
\[
  u_t+a u_x=0,
  \qquad a>0,
\]
on a periodic one-dimensional domain.  Let \(h\) be the fine-cell width,
\(\Delta t\) the time step, and
\[
  \lambda=\frac{a\Delta t}{h}
\]
the Courant number.  The first-order upwind finite-volume method with
forward Euler time integration is
\[
  u_i^{n+1}
  =
  (1-\lambda)u_i^n+\lambda u_{i-1}^n,
  \qquad 0<\lambda\le1,
\]
with indices understood periodically.  In the algebra below we use unit
fine-cell widths, which amounts to setting \(h=1\) in the restriction
operator.

Let \(P\ge2\) be the number of parent cells and let each parent contain
\(r\ge2\) fine cells.  The total number of fine cells is
\[
  N=Pr.
\]
Write the fine state as
\[
  x=(x_{K,j}),
  \qquad
  K=0,\ldots,P-1,
  \qquad
  j=0,\ldots,r-1.
\]
All parent indices are interpreted modulo \(P\), so that \(K=-1\) means
\(K=P-1\) and \(K=P\) means \(K=0\).

The parent average operator \(R:\R^{Pr}\to\R^P\) is
\[
  (Rx)_K=\frac1r\sum_{j=0}^{r-1}x_{K,j}.
\]
Let \(S\) be the downstream cyclic shift,
\[
 (Sx)_{K,j}
 =
 \begin{cases}
   x_{K,j-1}, & j\ge1,\\
   x_{K-1,r-1}, & j=0.
 \end{cases}
\]
The positive upwind forward-Euler step is
\[
  A_\lambda=(1-\lambda)I+\lambda S,
  \qquad
  0<\lambda\le1.
\]
For prediction horizon \(L\ge0\), define the observation matrix
\[
  {\cal O}_L =
  \begin{bmatrix}
  R\\
  RA_\lambda\\
  \vdots\\
  RA_\lambda^L
  \end{bmatrix}.
\]
The matrix \({\cal O}_L\) is a finite-horizon observation matrix.  It maps
one initial fine-grid state to the sequence of parent averages observed from
time \(0\) through time \(L\).  Each block \(RA_\lambda^t\) has size
\(P\times Pr\), so
\[
  {\cal O}_L\in\mathbb R^{(L+1)P\times Pr}.
\]
Thus every additional observed time step adds \(P\) rows, one for each
parent average at that time.  The number of rows grows with \(L\), but the
rank need not.  In the upwind hierarchy studied here, the rank grows only
until the exposed horizon reaches \(r-1\) child layers, after which later
rows are linear combinations of earlier rows.

The vector \({\cal O}_Lx\) is the sequence of parent averages generated by
the same realized fine state \(x\) from time \(0\) through time \(L\).

We use three notions throughout the paper.  First, \emph{exact memory}
means the additional real-valued information, beyond the current parent
averages, required to reconstruct the parent-average history
\({\cal O}_Lx\) for every fine state in a declared support.  Second,
\emph{stable observability} refers to whether the algebraically visible
directions are visible above a declared singular-value tolerance in
floating-point arithmetic.  Third, \emph{dissipative erasure} means decay of
the nonconstant component of the computed solution, relative to its computed
mean, below a prescribed numerical tolerance or perturbation floor.  These
notions are distinct: exact memory is an algebraic dimension question,
stable observability is a conditioning question, and dissipative erasure is
a long-time decay question.

\section{Exact finite-horizon memory} \label{sec:exact_fh}

We first state a standard dimension lower bound.  It converts an observation
rank into a lower bound on any continuous real-valued encoder. For an open support \(\Omega\subset\mathbb R^{Pr}\), define the centralized
finite-horizon exact excess memory dimension by
\[
  d^{\rm cen}_L(\Omega)
  =
  \inf
  \left\{
   d\in\mathbb N_0:
  \begin{array}{l}
  \text{there exist a continuous }M:\Omega\to\mathbb R^d
  \text{ and a decoder }G\\
  \text{such that }
  {\cal O}_Lx=G(Rx,M(x))\text{ for every }x\in\Omega
  \end{array}
  \right\}.
\]
This quantity counts real-valued coordinates beyond the parent averages.  It
is not a bit-complexity, communication-cost, or statistical learning
quantity.

For a product-local architecture, assume
\[
  \Omega=\Omega_0\times\cdots\times\Omega_{P-1},
  \qquad
  \Omega_K\subset\mathbb R^r\text{ open}.
\]
A vector \(m=(m_0,\ldots,m_{P-1})\) is called locally admissible at horizon
\(L\) if there exist continuous parent-local maps
\[
  M_K:\Omega_K\to\mathbb R^{m_K},
  \qquad K=0,\ldots,P-1,
\]
and a decoder \(G\) such that
\[
  {\cal O}_Lx
  =
  G\bigl(Rx,M_0(x_0),\ldots,M_{P-1}(x_{P-1})\bigr)
  \qquad
  \text{for every }x\in\Omega.
\]
The terms encoder and decoder are used here in a finite-dimensional
mathematical sense, not in the Shannon-information sense.  An encoder is an
auxiliary memory map \(M\) that extracts additional real-valued coordinates
from the fine state, beyond the parent averages.  A decoder is a prediction
map \(G\) that uses the parent averages and this auxiliary memory to
reconstruct the desired parent-average history.  The lower bounds below
concern Euclidean state dimension for continuous maps on open sets; they are
not bit-complexity, entropy, mutual-information, or channel-capacity
statements.  Continuity is required of the encoder; the decoder is only
assumed to be defined on the attained encoded data.

The theorem below determines both the centralized dimension and the
coordinatewise local lower bound.

\begin{lemma}[Continuous encoder lower bound]
\label{lem:encoder}
Let \(H:\mathbb R^n\to\mathbb R^m\) be linear, and let
\(\Omega\subset\mathbb R^n\) contain a nonempty open set.  Let
\(C:\mathbb R^n\to\mathbb R^p\) be linear.  Suppose there are a continuous
map \(M:\Omega\to\mathbb R^d\) and a decoder \(G\), defined on the attained
image, such that
\[
  Hx=G(Cx,M(x)),
  \qquad x\in\Omega.
\]
Then
\[
  \rank C+d\ge \rank H.
\]
\end{lemma}

\begin{proof}
Let \(k=\rank H\) and \(c=\rank C\).  Replace \(C\) by coordinates on its
range, so that \(Cx\) is viewed as an element of a \(c\)-dimensional linear
space.  Choose a \(k\)-dimensional linear complement \(V\) of \(\ker H\).  On
a small relative-open ball in an affine copy of \(V\), the map \(H\) is
injective.  If two points in that ball have the same \((C,M)\) value, then
the decoder gives them the same \(H\) value, and hence the two points are
equal.  Thus \((C,M)\) continuously injects an open subset of \(\mathbb R^k\) into \(\mathbb R^{c+d}\).  By invariance of domain, \(c+d\ge k\).
\end{proof}

\begin{theorem}[Finite-horizon upwind memory]
\label{thm:rank}
Let \(P\ge2\), \(r\ge2\), \(L\ge0\), and \(0<\lambda\le1\).  Put
\[
  q=\min\{L,r-1\}.
\]
Then
\[
  \rank {\cal O}_L
  =
  P+(P-1)q.
\]
Consequently, the following encoder lower bounds hold.  First, any
continuous centralized encoder that reconstructs \({\cal O}_Lx\) from \(Rx\)
on an open set requires at least
\[
  (P-1)q
\]
additional real coordinates.

Second, suppose the support has product form
\[
  \Omega=\Omega_0\times\cdots\times\Omega_{P-1}
  \subset (\mathbb R^r)^P,
\]
where each \(\Omega_K\subset\mathbb R^r\) is open, and suppose the message is
a product of continuous parent-local producers,
\[
  M(x)=\bigl(M_0(x_0),\ldots,M_{P-1}(x_{P-1})\bigr),
  \qquad
  M_K:\Omega_K\to\mathbb R^{m_K},
\]
with
\[
  x_K=(x_{K,0},\ldots,x_{K,r-1}).
\]
If a decoder reconstructs \({\cal O}_Lx\) from \(Rx\) and \(M(x)\) for all
\(x\in\Omega\), then
\[
  m_K\ge q
  \qquad
  \text{for every }K=0,\ldots,P-1.
\]

Moreover, the centralized lower bound is attained by the anchored flux-divergence queue
\[
  \bigl(JB\Phi_0,\ldots,JB\Phi_{q-1}\bigr)
  \in\R^{(P-1)q},
\]
where
\[
  \Phi_{K,t}=\lambda(A_\lambda^t x)_{K,r-1},
  \qquad
  (B\Phi_t)_K=\Phi_{K,t}-\Phi_{K-1,t},
\]
and \(J\) deletes one coordinate of a zero-sum periodic vector.
\end{theorem}

\begin{proof}
Since \(I\) and \(S\) commute,
\[
  A_\lambda^t
  =
  \sum_{j=0}^{t}
  \binom{t}{j}
  (1-\lambda)^{t-j}\lambda^j S^j .
\]
The change of variables from
\((R,RS,\ldots,RS^L)\) to
\((R,RA_\lambda,\ldots,RA_\lambda^L)\) is triangular with nonzero diagonal
entries \(1,\lambda,\ldots,\lambda^L\).  Therefore these two families have
the same row span.

For \(0\le t\le r-2\), define
\[
  D_t=r(RS^{t+1}-RS^t).
\]
Then
\[
  (D_t x)_K
  =
  x_{K-1,r-1-t}-x_{K,r-1-t}.
\]
Thus each new shift exposes one new child layer through the periodic
difference operator.  The \(P\) rows of \(D_t\) sum to zero, and hence
\[
  \rank D_t\le P-1.
\]
Moreover, for \(t\ge1\),
\[
  RS^t
  =
  R+\frac1r\sum_{s=0}^{t-1}D_s.
\]
Therefore, for \(L\le r-1\),
\[
  \rank\operatorname{span}\{R,RS,\ldots,RS^L\}
  \le
  P+L(P-1).
\]
If \(L>r-1\), then \(S^r\) shifts by one parent cell, so later rows
\(RS^t\), \(t\ge r\), lie in the row span of
\[
  R,RS,\ldots,RS^{r-1}.
\]
Consequently, for all \(L\ge0\),
\[
  \rank\operatorname{span}\{R,RS,\ldots,RS^L\}
  \le
  P+q(P-1),
  \qquad
  q=\min\{L,r-1\}.
\]
We now construct a matching independent set of rows.

Let
\[
  j_\ell=r-1-\ell,
  \qquad
  \ell=0,\ldots,q-1,
\]
be the \(q\) exposed right-collar layers.  Since \(q\le r-1\), none of these
layers is \(j=0\).  For each exposed layer \(j_\ell\), take the \(P-1\)
difference rows
\[
  D_{\ell,K}
  =
  r(RS^{\ell+1}-RS^\ell)_K,
  \qquad
  K=1,\ldots,P-1,
\]
and take the \(P\) parent-mean rows \(R_K\), \(K=0,\ldots,P-1\).  Consider
the columns consisting of:
\[
  (K,j_\ell),\qquad K=1,\ldots,P-1,\quad \ell=0,\ldots,q-1,
\]
together with the \(P\) columns
\[
  (K,0),\qquad K=0,\ldots,P-1.
\]
Order the difference rows first, layer by layer, and the parent-mean rows
last.  Order the columns in the same way, with the child-\(0\) columns last.
The resulting square submatrix has block form
\[
  \begin{bmatrix}
  D & 0\\
  * & r^{-1}I_P
  \end{bmatrix}.
\]
Here \(D\) is block diagonal across the exposed layers.  In each layer, the
\((P-1)\times(P-1)\) difference block is lower triangular with diagonal
entries \(-1\).  Hence
\[
  \det D=\pm1,
  \qquad
  \det
  \begin{bmatrix}
  D & 0\\
  * & r^{-1}I_P
  \end{bmatrix}
  =
  \pm r^{-P}\ne0.
\]
Therefore the parent means and the \(q\) exposed layer-difference blocks are
linearly independent.  They give
\[
  P+(P-1)q
\]
independent rows.

No later shift adds rank.  Indeed, \(S^r\) shifts by one parent cell, so
\(RS^{t+r}\) is a parent-level cyclic permutation of \(RS^t\).  This proves
\[
  \rank{\cal O}_L=P+(P-1)q.
\]

The centralized encoder lower bound follows from \Cref{lem:encoder} with
\(H={\cal O}_L\), \(C=R\), and \(p=P\).

We next prove the product-local lower bound.  If \(q=0\), there is nothing
to prove, so assume \(q\ge1\).  Fix a parent \(K\).  Choose a point
\(\bar x\in\Omega\).  Since \(\Omega_K\) is open, there is a sufficiently
small open ball \(B\subset\mathbb R^q\) such that the following
\(q\)-parameter family remains in \(\Omega_K\):
\[
  x_{K,0}(\xi)
  =
  \bar x_{K,0}
  -
  \sum_{\ell=0}^{q-1}\xi_\ell,
  \qquad
  x_{K,r-1-\ell}(\xi)
  =
  \bar x_{K,r-1-\ell}
  +
  \xi_\ell,
  \quad
  \ell=0,\ldots,q-1,
\]
with all other children of parent \(K\) fixed at their values in \(\bar x\).
All other parents are also fixed at their values in \(\bar x\).  The
compensation in child \(0\) keeps the parent average of parent \(K\) fixed,
so \(Rx(\xi)\) is independent of \(\xi\).  The messages from all parents
other than \(K\) are also independent of \(\xi\).

Suppose two parameters \(\xi,\xi'\in B\) give the same local message:
\[
  M_K(x_K(\xi))=M_K(x_K(\xi')).
\]
Then the full data \((Rx,M(x))\) are the same for \(x(\xi)\) and
\(x(\xi')\).  Since the decoder reconstructs \({\cal O}_Lx\), we must have
\[
  {\cal O}_L x(\xi)={\cal O}_L x(\xi').
\]
Let
\[
  d=x(\xi)-x(\xi'),
  \qquad
  \delta_\ell=\xi_\ell-\xi'_\ell.
\]
The triangular relation between
\((R,RS,\ldots,RS^L)\) and
\((R,RA_\lambda,\ldots,RA_\lambda^L)\) is invertible.  Hence
\({\cal O}_Ld=0\) implies
\[
  RS^t d=0,
  \qquad
  0\le t\le L.
\]
For each \(\ell=0,\ldots,q-1\), we have \(\ell+1\le L\).  Using
\[
  r(RS^{\ell+1}d-RS^\ell d)_K
  =
  d_{K-1,r-1-\ell}-d_{K,r-1-\ell},
\]
and noting that only parent \(K\) varies, we obtain
\[
  0
  =
  r(RS^{\ell+1}d-RS^\ell d)_K
  =
  -\delta_\ell.
\]
Thus \(\delta_\ell=0\) for every \(\ell\), so \(\xi=\xi'\).  Therefore the
map
\[
  \xi\mapsto M_K(x_K(\xi))
\]
is a continuous injection from an open subset of \(\mathbb R^q\) into
\(\mathbb R^{m_K}\).  By invariance of domain, \(m_K\ge q\).  Since \(K\) was
arbitrary, this holds for every parent.

We also record the local raw outflow map, since it gives the attaining local
queue and will be used later for conditioning.  Let
\[
  c_{K,j}=x_{K,r-1-j},
  \qquad
  j=0,\ldots,q-1,
\]
be the right-collar values of parent \(K\).  Before wraparound, the outgoing
flux queue satisfies
\[
  \Phi_{K,t}
  =
  \lambda
  \sum_{j=0}^{t}
  \binom{t}{j}
  (1-\lambda)^{t-j}\lambda^j c_{K,j},
  \qquad
  t=0,\ldots,q-1.
\]
Thus \((\Phi_{K,0},\ldots,\Phi_{K,q-1})\) is obtained from the collar
\((c_{K,0},\ldots,c_{K,q-1})\) by a lower-triangular matrix with diagonal
entries
\[
  \lambda,\lambda^2,\ldots,\lambda^q.
\]
For \(0<\lambda\le1\), this map is invertible.  Hence the raw \(q\)-entry
local outflow queue contains exactly the exposed \(q\) collar coordinates.

It remains to show attainment.  The parent update can be written as
\[
  RA_\lambda^{t+1}x
  =
  RA_\lambda^t x
  -
  \frac1r B\Phi_t .
\]
The vector \(B\Phi_t\) has zero sum, so deleting one coordinate loses no
information.  Given \(Rx\) and the anchored sequence
\((JB\Phi_0,\ldots,JB\Phi_{q-1})\), one reconstructs the full divergence
sequence and hence the parent history through the exposed horizon.

If \(L\le r-1\), this is the entire required history.  If \(L>r-1\), then
\(q=r-1\), and the same saturated queue reconstructs
\[
  y_0,y_1,\ldots,y_{r-1},
  \qquad
  y_n=RA_\lambda^n x.
\]
The later parent states are then determined by a fixed linear recurrence.
Indeed,
\[
  A_\lambda-(1-\lambda)I=\lambda S,
\]
so
\[
  \bigl(A_\lambda-(1-\lambda)I\bigr)^r
  =
  \lambda^r S^r.
\]
Let \(\Pi_P\) denote the parent-level cyclic shift,
\[
  (\Pi_P y)_K=y_{K-1}.
\]
Multiplying the previous identity by \(RA_\lambda^n\) gives
\[
  y_{n+r}
  =
  \lambda^r \Pi_P y_n
  -
  \sum_{j=0}^{r-1}
  \binom{r}{j}
  \bigl(-(1-\lambda)\bigr)^{r-j}
  y_{n+j}.
\]
Therefore \(y_r,y_{r+1},\ldots,y_L\) are obtained recursively from
\(y_0,\ldots,y_{r-1}\).  Hence the saturated anchored queue determines
\({\cal O}_Lx\) for every \(L\ge r-1\).
\end{proof}
In the notation above, \Cref{thm:rank} gives
\[
  d^{\rm cen}_L(\Omega)
  =
  (P-1)\min\{L,r-1\}
\]
for every open support \(\Omega\) on which the stated encoder problem is
posed.  For product-local encoders, every admissible vector
\((m_0,\ldots,m_{P-1})\) satisfies
\[
  m_K\ge \min\{L,r-1\},
  \qquad K=0,\ldots,P-1,
\]
and the raw local outflow queues attain this coordinatewise bound.

\begin{corollary}[All-horizon invisible subspace]
\label{cor:invisible-subspace}
For \(L\ge r-1\), the nullspace of the saturated observation operator is
\[
  {\cal N}_{\infty}
  =
  \left\{
  x\in\mathbb R^{Pr}:
  x_{K,j}=c_j\ \text{for every }K,\quad
  \sum_{j=0}^{r-1}c_j=0
  \right\}.
\]
Consequently,
\[
  \dim {\cal N}_{\infty}=r-1.
\]
These are precisely the fine-grid directions that remain invisible to all
future parent-average observations.
\end{corollary}

\begin{proof}
If \(x_{K,j}=c_j\) for every parent \(K\), then every parent contains the
same child profile.  Since \(\sum_{j=0}^{r-1}c_j=0\), all parent averages are
zero.  Since \(A_\lambda\) is a polynomial in the periodic shift \(S\), each
future state is a linear combination of shifted copies of the same repeated
zero-mean profile.  Hence all future parent averages remain zero.

Conversely, for \(L\ge r-1\), \Cref{thm:rank} gives
\[
  \rank {\cal O}_L
  =
  P+(P-1)(r-1)
  =
  Pr-r+1.
\]
Thus the nullity is
\[
  Pr-(Pr-r+1)=r-1.
\]
The displayed invisible subspace has dimension \(r-1\), because it is the
space of \(r\)-component profiles with zero sum.  Therefore it is the entire
saturated nullspace.
\end{proof}

\begin{corollary}[Sharp finite delay]
\label{cor:sharp-delay}
For every \(m\in\{0,\ldots,r-2\}\), there exist two fine-grid states whose
parent averages agree from time \(0\) through time \(m\), but differ at time
\(m+1\).  Conversely, if two fine-grid states have identical parent-average
histories through time \(r-1\), then their parent-average histories agree for
all subsequent times.
\end{corollary}

\begin{proof}
Fix a parent \(K\), choose \(\delta\ne0\), and set
\[
  \ell=r-1-m.
\]
Consider the difference vector
\[
  v=\delta(e_{K,0}-e_{K,\ell}).
\]
The two nonzero entries of \(v\) lie in the same parent and have zero parent
average.  For \(0\le t\le m\), every shift \(S^s v\) appearing in
\(A_\lambda^t v\), with \(0\le s\le t\), still has both nonzero entries
inside the same parent.  Hence
\[
  RA_\lambda^t v=0,
  \qquad
  0\le t\le m.
\]
At time \(m+1\), the component initially at child \(\ell\) reaches the next
parent.  The term with \(S^{m+1}\) has coefficient \(\lambda^{m+1}\), so
\[
  RA_\lambda^{m+1}v
  =
  \frac{\lambda^{m+1}\delta}{r}
  (e_K-e_{K+1})
  \ne0,
\]
with parent indices interpreted periodically.  This proves the sharp delayed
separation.

The converse follows from the recurrence in the proof of \Cref{thm:rank}.
Once the parent-average history through time \(r-1\) is fixed, the saturated
recurrence determines every later parent average.
\end{proof}

\begin{corollary}[Saturated memory and periodic gauge]
\label{cor:gauge}
At saturation, \(q=r-1\).  A product-local upstream encoder needs \(r-1\)
coordinates per parent: each parent must retain its unresolved downstream,
or outflow, collar.  A centralized encoder needs only
\[
  (P-1)(r-1)
\]
coordinates beyond the parent means.  The saving of \(r-1\) coordinates is
exactly the spatially uniform periodic flux gauge, one redundant constant
face-flux mode at each queue time.
\end{corollary}

\begin{proof}
Set \(q=r-1\) in \Cref{thm:rank}.  The product-local lower bound gives
\(m_K\ge r-1\) for every parent \(K\).  The local raw outflow queue has
exactly \(r-1\) entries and, by the triangular collar-to-queue map in the
proof of \Cref{thm:rank}, it recovers the \(r-1\) right-collar values of the
parent.  The remaining child value is then determined by the parent average.

For the centralized encoder, the anchored divergence queue has
\((P-1)(r-1)\) coordinates.  At each queue time, adding the same constant
face flux to every periodic parent interface produces no change in the
parent-average update, since only cyclic flux differences enter the update.
This removes exactly one scalar gauge coordinate per queue time, hence
exactly \(r-1\) coordinates at saturation.
\end{proof}

\begin{corollary}[Autonomous saturated predictive state]
\label{cor:autonomous-state}
Let
\[
  y_n=RA_\lambda^n x
\]
and define the saturated state
\[
  Z_n
  =
  \left(
  y_n,
  JB\Phi_n,
  JB\Phi_{n+1},
  \ldots,
  JB\Phi_{n+r-2}
  \right).
\]
Then there exists a fixed linear map \({\cal F}\), depending only on
\((P,r,\lambda)\), such that
\[
  Z_{n+1}={\cal F}Z_n
\]
for every \(n\ge0\).  Thus the parent means together with the saturated
flux-divergence queue form an autonomous predictive state and advance without
revisiting the fine-grid state.
\end{corollary}

\begin{proof}
Given \(Z_n\), the anchored divergence entries reconstruct
\[
  B\Phi_n,\ldots,B\Phi_{n+r-2},
\]
because each \(B\Phi_t\) has zero sum.  Starting from \(y_n\), the update
\[
  y_{t+1}
  =
  y_t-\frac1r B\Phi_t
\]
then reconstructs
\[
  y_{n+1},\ldots,y_{n+r-1}.
\]
The recurrence from the proof of \Cref{thm:rank},
\[
  y_{n+r}
  =
  \lambda^r \Pi_P y_n
  -
  \sum_{j=0}^{r-1}
  \binom{r}{j}
  \bigl(-(1-\lambda)\bigr)^{r-j}
  y_{n+j},
\]
gives \(y_{n+r}\).  Finally,
\[
  B\Phi_{n+r-1}
  =
  r(y_{n+r-1}-y_{n+r})
\]
provides the new tail of the queue.  Every operation is linear and depends
only on \((P,r,\lambda)\), so the update has the form
\(Z_{n+1}={\cal F}Z_n\).
\end{proof}

\begin{corollary}[Minimal autonomous predictive-state dimension]
\label{cor:minimal-autonomous-dimension}
Let \(\Omega\subset\mathbb R^{Pr}\) contain a nonempty open set.  Let
\(E:\Omega\to\mathbb R^d\) be a continuous encoder.  Suppose there exist maps
\[
  F:\mathbb R^d\to\mathbb R^d,
  \qquad
  H_{\rm out}:\mathbb R^d\to\mathbb R^P,
\]
such that
\[
  RA_\lambda^n x
  =
  H_{\rm out}\bigl(F^n(E(x))\bigr)
  \qquad
  \text{for every }x\in\Omega\text{ and every }n\ge0.
\]
Then
\[
  d\ge \rank{\cal O}_{r-1}
  =
  Pr-r+1.
\]
The saturated state in \Cref{cor:autonomous-state} attains this dimension.
\end{corollary}

\begin{proof}
The assumed autonomous representation determines
\[
  RA_\lambda^n x,
  \qquad
  n=0,\ldots,r-1,
\]
from \(E(x)\).  Hence \({\cal O}_{r-1}x\) factors through the continuous map
\(E:\Omega\to\mathbb R^d\).  Applying \Cref{lem:encoder} with
\(H={\cal O}_{r-1}\) and \(C=0\) gives
\[
  d\ge \rank{\cal O}_{r-1}.
\]
By \Cref{thm:rank},
\[
  \rank{\cal O}_{r-1}
  =
  P+(P-1)(r-1)
  =
  Pr-r+1.
\]
The state in \Cref{cor:autonomous-state} has exactly
\(P+(P-1)(r-1)=Pr-r+1\) coordinates, so the bound is sharp.
\end{proof}

\section{Stable observability and effective rank}
\label{sec:stable-observability}

The exact rank theorem is algebraic.  It says which directions are visible
in exact arithmetic.  It does not say whether those directions are stably
recoverable from floating-point data.  This distinction is standard in
numerical linear algebra \cite{TrefethenBau:1997nla,Higham:2002accuracy} and is closely related to observability questions for discretized partial differential equations (PDEs) \cite{CohnDee:1988observability}.  Classical realization and model-reduction theory also connect finite input-output
histories, observability, singular values, and minimal state representations
\cite{HoKalman:1966realization,Moore:1981pca}.

Let
\[
  c_j=x_{K,r-1-j}
\]
denote the right collar values of one parent.  Before wraparound, the
outgoing flux queue satisfies
\[
  \Phi_t
  =
  \lambda
  \sum_{j=0}^{t}
  \binom{t}{j}
  (1-\lambda)^{t-j}\lambda^j c_j,
  \qquad
  t=0,\ldots,q-1.
\]
This defines a lower-triangular matrix \(T_q\) mapping the \(q\) collar
values to the first \(q\) outflow values.

\begin{theorem}[Conditioning of delayed queue recovery]
\label{thm:conditioning}
For \(q\ge1\), let \(T_q\) be the collar-to-queue matrix defined by
\[
  (T_q)_{tj}
  =
  \binom{t}{j}(1-\lambda)^{t-j}\lambda^{j+1},
  \qquad
  0\le j\le t\le q-1,
\]
and \((T_q)_{tj}=0\) for \(j>t\).  Then \(T_q\) is lower triangular with
diagonal entries
\[
  \lambda,\lambda^2,\ldots,\lambda^q,
\]
and
\[
  \det T_q=\lambda^{q(q+1)/2}.
\]
Its inverse is also lower triangular and is given by
\[
  (T_q^{-1})_{tj}
  =
  \lambda^{-(t+1)}
  \binom{t}{j}
  \bigl[-(1-\lambda)\bigr]^{t-j},
  \qquad
  0\le j\le t\le q-1.
\]
Moreover,
\[
  \|T_q\|_\infty=\lambda,
  \qquad
  \|T_q^{-1}\|_\infty
  =
  \lambda^{-q}(2-\lambda)^{q-1},
\]
so
\[
  \kappa_\infty(T_q)
  =
  \left(\frac{2-\lambda}{\lambda}\right)^{q-1}.
\]
Finally,
\[
  \frac{\lambda^q}{\sqrt q\,(2-\lambda)^{q-1}}
  \le
  \sigma_{\min}(T_q)
  \le
  \lambda^q.
\]
Thus, for fixed \(q\), the smallest singular value scales like
\(\lambda^q\) as \(\lambda\to0\).
\end{theorem}

\begin{proof}
The triangular form and diagonal entries follow immediately from the
definition of \(T_q\).  The determinant is the product of the diagonal
entries.

The stated inverse follows from the binomial inversion identity.  Direct
multiplication gives, for \(j\le t\),
\[
  \sum_{m=j}^{t}
  \binom{t}{m}(1-\lambda)^{t-m}\lambda^{m+1}
  \lambda^{-(m+1)}
  \binom{m}{j}\bigl[-(1-\lambda)\bigr]^{m-j}
  =
  \binom{t}{j}(1-\lambda)^{t-j}
  \sum_{m=j}^{t}
  \binom{t-j}{m-j}(-1)^{m-j}.
\]
The final sum is \(1\) when \(t=j\) and \(0\) otherwise, proving the inverse
formula.

For the infinity norm,
\[
  \sum_{j=0}^{t}
  |(T_q)_{tj}|
  =
  \lambda
  \sum_{j=0}^{t}
  \binom{t}{j}(1-\lambda)^{t-j}\lambda^j
  =
  \lambda.
\]
Similarly,
\[
  \sum_{j=0}^{t}
  |(T_q^{-1})_{tj}|
  =
  \lambda^{-(t+1)}
  \sum_{j=0}^{t}
  \binom{t}{j}(1-\lambda)^{t-j}
  =
  \lambda^{-(t+1)}(2-\lambda)^t,
\]
whose maximum over \(t=0,\ldots,q-1\) is
\(\lambda^{-q}(2-\lambda)^{q-1}\).  This proves the formula for
\(\kappa_\infty(T_q)\).

The upper bound on \(\sigma_{\min}(T_q)\) follows by applying \(T_q\) to the
last coordinate vector:
\[
  \sigma_{\min}(T_q)\le \|T_q e_{q-1}\|_2=\lambda^q.
\]
For the lower bound, use
\[
  \sigma_{\min}(T_q)=\|T_q^{-1}\|_2^{-1}
\]
and
\[
  \|T_q^{-1}\|_2\le \sqrt q\,\|T_q^{-1}\|_\infty.
\]
This gives the stated two-sided estimate.
\end{proof}

\begin{proposition}[Spatial conditioning of the periodic divergence]
\label{prop:spatial-conditioning}
Let \(B:\mathbb R^P\to\mathbb R^P\) be the periodic difference operator
\[
  (B\phi)_K=\phi_K-\phi_{K-1}.
\]
Then \(B\mathbf 1=0\), and the nonzero singular values of \(B\) are
\[
  \sigma_k(B)
  =
  2\left|\sin\left(\frac{\pi k}{P}\right)\right|,
  \qquad
  k=1,\ldots,P-1.
\]
In particular,
\[
  \sigma_{\min}^+(B)
  =
  2\sin\left(\frac{\pi}{P}\right),
\]
where \(\sigma_{\min}^+\) denotes the smallest nonzero singular value.
Thus centralized recovery through flux differences has both a temporal
conditioning mechanism, represented by \(T_q\), and a spatial conditioning
mechanism, represented by \(B\).
\end{proposition}

\begin{proof}
The operator \(B\) is circulant.  Its Fourier eigenvalues are
\[
  1-e^{-2\pi i k/P},
  \qquad
  k=0,\ldots,P-1.
\]
The corresponding singular values are their moduli,
\[
  |1-e^{-2\pi i k/P}|
  =
  2\left|\sin\left(\frac{\pi k}{P}\right)\right|.
\]
The \(k=0\) mode is the constant vector and gives the periodic flux gauge.
The smallest nonzero value occurs at \(k=1\) and \(k=P-1\).
\end{proof}

For algebraic dimension counting, deleting one coordinate of a zero-sum
periodic vector is sufficient.  For conditioning analysis, an orthonormal
basis of \(\mathbf 1^\perp\) gives a symmetric representation of the same
gauge quotient.  The experiments use the algebraic quotient, while
\Cref{prop:spatial-conditioning} records the intrinsic conditioning of the
periodic difference itself.

We use exact equalities for exact arithmetic and a separate tolerance model
for floating-point arithmetic.  In the numerical experiments we use a tolerance-dependent effective rank computed from a singular value decomposition (SVD).  Since
\[
  {\cal O}_L\in\mathbb R^{(L+1)P\times Pr},
\]
and since \(s_i({\cal O}_L)\) denotes the singular values of
\({\cal O}_L\), counted with multiplicity, we define
\[
  \rank_\tau({\cal O}_L)
  =
  \card\{i:s_i({\cal O}_L)>\tau\},
  \qquad
  \tau
  =
  C_{\rm svd}\epsmach
  \max\{(L+1)P,Pr\}
  \|{\cal O}_L\|_2 .
\]
Here \(C_{\rm svd}=100\) in all reported computations.  The factor
\(\max\{(L+1)P,Pr\}\) follows the usual matrix-size scaling used in
singular-value based numerical-rank decisions
\cite{TrefethenBau:1997nla,Higham:2002accuracy}.  This definition is not a
new mathematical rank; it is a declared floating-point diagnostic.

This is the appropriate way to account for machine precision.  Adding an
informal ``\(+\epsmach\)'' to every identity would obscure the difference
between exact algebra and numerical tolerance.

\section{Dissipative erasure} \label{sec:diss_eras}

For \(0<\lambda<1\), first-order upwind is dissipative.  This changes the
practical meaning of missing information.  A hidden subcell difference may
be required for exact short-time closure, yet later be damped below a
declared perturbation floor.

The result in this section is a bounded-perturbation theorem.  It does not
derive a complete implementation-specific Institute of Electrical and
Electronics Engineers (IEEE) floating-point roundoff model.  Instead, it
assumes that the computed update is the exact upwind update plus a
per-step perturbation with separately bounded mean and nonconstant
components.  This is the level of modeling used in the numerical
interpretation below.

At the level of modified-equation analysis
\cite{warming1974,harten1983,tadmor1987,MargolinRider}, the same dissipative mechanism
appears as an artificial diffusion term.  For
\[
  \lambda=\frac{a\Delta t}{h},
\]
the first-order upwind update has the formal modified equation
\[
  u_t+a u_x
  =
  \frac{ah}{2}(1-\lambda)u_{xx}
  +
  {\cal O}(h^2).
\]
Thus the contraction per update and the numerical diffusion per unit
physical time are related but distinct quantities.  This also explains why
the control case \(\lambda=1\) is nondissipative: the leading artificial
diffusion coefficient vanishes.

Let \(P_0\) denote the projection onto the spatial mean,
\[
  P_0 z=\bar z\mathbf 1,
  \qquad
  \bar z=\frac1N\sum_{i=0}^{N-1}z_i.
\]
In this section \(z^n\in\mathbb R^N\) denotes the evolving fine-grid solution
vector; it is the same type of object denoted by \(u^n\) in the numerical
experiments.

\begin{theorem}[Dissipative decay under bounded arithmetic perturbations]
\label{thm:erasure}
Let \(0<\lambda<1\), and suppose the computed solution satisfies
\[
  \tilde z^{n+1}
  =
  A_\lambda \tilde z^n+\delta^n,
\]
with
\[
  \|(I-P_0)\delta^n\|_2\le\eta,
  \qquad
  |\overline{\delta^n}|\le\eta_m.
\]
Here
\[
  \overline{\delta^n}
  =
  \frac1N\mathbf 1^\top\delta^n
\]
is the spatial mean of the arithmetic perturbation at step \(n\).

Define
\[
  \rho_N
  =
  \sqrt{
  1-4\lambda(1-\lambda)\sin^2\left(\frac{\pi}{N}\right)
  }.
\]
Then \(0<\rho_N<1\), and
\[
  \|(I-P_0)\tilde z^n\|_2
  \le
  \rho_N^n\|(I-P_0)\tilde z^0\|_2
  +
  \frac{1-\rho_N^n}{1-\rho_N}\eta .
\]
The computed mean satisfies
\[
  |\bar{\tilde z}^{\,n}-\bar{\tilde z}^{\,0}|
  \le
  n\eta_m .
\]
Finally,
\[
  \|R\tilde z^n-\bar{\tilde z}^{\,n}\mathbf 1_P\|_2
  \le
  \frac1{\sqrt r}
  \left[
  \rho_N^n\|(I-P_0)\tilde z^0\|_2
  +
  \frac{1-\rho_N^n}{1-\rho_N}\eta
  \right].
\]
\end{theorem}

\begin{proof}
First note the elementary dissipation identity.  Since the cyclic shift
\(S\) is orthogonal,
\[
  \|A_\lambda z\|_2^2
  =
  \|z\|_2^2
  -
  \lambda(1-\lambda)\|(I-S)z\|_2^2.
\]
Thus the scheme is nondissipative at \(\lambda=1\), while for \(0<\lambda<1\) every nonconstant component loses discrete \(L^2\) energy.
The Fourier argument below gives the sharp contraction rate on the zero-mean subspace.

The matrix \(A_\lambda\) is circulant and normal.  Its Fourier eigenvalues
are
\[
  \mu_k=(1-\lambda)+\lambda e^{-2\pi i k/N},
  \qquad
  k=0,\ldots,N-1.
\]
For \(k\ne0\),
\[
  |\mu_k|^2
  =
  1-4\lambda(1-\lambda)\sin^2\left(\frac{\pi k}{N}\right)
  \le
  \rho_N^2.
\]
Hence
\[
  \|A_\lambda v\|_2\le \rho_N\|v\|_2
  \qquad
  \hbox{whenever }P_0 v=0.
\]
Set \(v^n=(I-P_0)\tilde z^n\).  Then
\[
  \|v^{n+1}\|_2
  \le
  \rho_N\|v^n\|_2+\eta .
\]
Iteration gives the first estimate.  Averaging the perturbation equation
gives the mean-drift estimate.  The restriction estimate follows from
Cauchy's inequality:
\[
  \|Rv\|_2\le r^{-1/2}\|v\|_2 .
\]
\end{proof}

\begin{corollary}[Step-function flattening]
Let
\[
  z_i^0 =
  \begin{cases}
  1, & 0\le i<m,\\
  0, & m\le i<N,
  \end{cases}
  \qquad
  \alpha=\frac{m}{N}.
\]
Then
\[
  \|z^0-\alpha\mathbf 1\|_2
  =
  \sqrt{N\alpha(1-\alpha)}.
\]
Under the bounded-perturbation model of \Cref{thm:erasure},
\[
  \max_i|\tilde z_i^n-\alpha|
  \le
  \rho_N^n\sqrt{N\alpha(1-\alpha)}
  +
  \frac{1-\rho_N^n}{1-\rho_N}\eta
  +
  n\eta_m .
\]
Thus the periodic upwind solution flattens about its computed mean up to a
declared perturbation floor.  It remains close to the initial mean \(\alpha\)
only to the extent quantified by the mean-drift term \(n\eta_m\).
\end{corollary}

\begin{proof}
The mean of the initial step is
\[
  \alpha=\frac1N\sum_{i=0}^{N-1}z_i^0=\frac{m}{N}.
\]
Hence
\[
  \|z^0-\alpha\mathbf 1\|_2^2
  =
  m(1-\alpha)^2+(N-m)\alpha^2.
\]
Using \(m=N\alpha\), this becomes
\[
  \|z^0-\alpha\mathbf 1\|_2^2
  =
  N\alpha(1-\alpha)^2+N(1-\alpha)\alpha^2
  =
  N\alpha(1-\alpha).
\]
Therefore
\[
  \|z^0-\alpha\mathbf 1\|_2
  =
  \sqrt{N\alpha(1-\alpha)}.
\]

By \Cref{thm:erasure},
\[
  \|\tilde z^n-\bar{\tilde z}^{\,n}\mathbf 1\|_2
  \le
  \rho_N^n\sqrt{N\alpha(1-\alpha)}
  +
  \frac{1-\rho_N^n}{1-\rho_N}\eta .
\]
Moreover,
\[
  |\bar{\tilde z}^{\,n}-\alpha|
  =
  |\bar{\tilde z}^{\,n}-\bar{\tilde z}^{\,0}|
  \le
  n\eta_m,
\]
assuming the computed initial state has mean \(\bar{\tilde z}^{\,0}=\alpha\).
For each component,
\[
  |\tilde z_i^n-\alpha|
  \le
  |\tilde z_i^n-\bar{\tilde z}^{\,n}|
  +
  |\bar{\tilde z}^{\,n}-\alpha|.
\]
Taking the maximum over \(i\), and using
\[
  \|\tilde z^n-\bar{\tilde z}^{\,n}\mathbf 1\|_\infty
  \le
  \|\tilde z^n-\bar{\tilde z}^{\,n}\mathbf 1\|_2,
\]
gives the stated estimate.
\end{proof}

\begin{remark}[Erasure time and grid dependence]
\label{rem:erasure-time}
Let
\[
  E_0=\|(I-P_0)\tilde z^0\|_2,
  \qquad
  E_{\rm floor}=\frac{\eta}{1-\rho_N}.
\]
If a tolerance \(\zeta>E_{\rm floor}\) is prescribed, then the estimate in \Cref{thm:erasure} guarantees
\[
  \|(I-P_0)\tilde z^n\|_2\le \zeta
\]
whenever
\[
  n
  \ge
  \frac{\log\left(E_0/(\zeta-E_{\rm floor})\right)}
       {-\log \rho_N}.
\]
Since
\[
  1-\rho_N
  \sim
  \frac{2\pi^2\lambda(1-\lambda)}{N^2}
  \qquad
  \text{as }N\to\infty,
\]
the erasure step count satisfies
\[
  n_{\rm erase}
  =
  O\left(
  \frac{N^2}{\lambda(1-\lambda)}
  \log\frac{E_0}{\zeta-E_{\rm floor}}
  \right),
\]
provided that \(\zeta>E_{\rm floor}\).  If \(E_0\) and \(\zeta-E_{\rm floor}\) remain bounded independently of \(N\), this reduces to the fixed-tolerance \(N^2/[\lambda(1-\lambda)]\) scaling.  If the floor \(E_{\rm floor}\) increases with \(N\), the condition \(\zeta>E_{\rm  floor}\) may eventually fail for a fixed tolerance.

The corresponding physical-time scaling is also grid dependent.  On a fixed periodic domain of length \(L_{\rm dom}\), with
\[
  h=\frac{L_{\rm dom}}{N},
  \qquad
  \Delta t=\frac{\lambda h}{a}
  =
  \frac{\lambda L_{\rm dom}}{aN},
\]
one obtains
\[
  t_{\rm erase}
  =
  n_{\rm erase}\Delta t
  =
  O\left(
  \frac{L_{\rm dom}N}{a(1-\lambda)}
  \log\frac{E_0}{\zeta-E_{\rm floor}}
  \right).
\]
Therefore, dissipative erasure is a long-time finite-grid effect, and the grid dependence must be stated explicitly.
\end{remark}

\begin{remark}
The endpoint \(\lambda=1\) is excluded from \Cref{thm:erasure}.  In that
case \(A_1=S\) is a unitary cyclic shift.  A step function translates without numerical flattening.
\end{remark}

\section{Numerical experiments}
\label{sec:num_exp}

The numerical experiments separate five effects and one consistency check.
The first experiment checks the exact rank ladder and its tolerance-dependent
effective rank.  The second displays the singular spectrum at the saturated
horizon, showing which exact directions are close to the floating-point
threshold.  The third tests the conditioning of the collar-to-queue map.  The fourth and fifth finite-volume simulations show dissipative flattening of a step function and a delayed same-parent-mean collision followed by erasure.  A
final scalar flux check verifies the one-step identity in \Cref{prop:flux-closure}.  All computations use double precision, and the SVD tolerance is declared explicitly.

\subsection{Common computational setup}

All experiments use the scalar periodic upwind update
\[
  u_i^{n+1}
  =
  (1-\lambda)u_i^n+\lambda u_{i-1}^n,
  \qquad i=0,\ldots,N-1,
\]
with periodic indexing.  The fine grid has
\[
  N=Pr
\]
cells, where \(P\) is the number of parent cells and \(r\) is the number of
fine cells per parent.  Fine cell \(i=Kr+j\) belongs to parent \(K\), with
\(j=0,\ldots,r-1\).

The parent-average restriction matrix \(R\in\mathbb R^{P\times Pr}\) is
defined by
\[
  R_{K,i}
  =
  \begin{cases}
  1/r, & Kr\le i\le Kr+r-1,\\
  0, & \text{otherwise}.
  \end{cases}
\]
The cyclic downstream shift matrix \(S\in\mathbb R^{N\times N}\) is defined
by
\[
  (Su)_i=u_{i-1},
\]
with periodic indexing.  The one-step upwind matrix is
\[
  A_\lambda=(1-\lambda)I+\lambda S.
\]
For a horizon \(L\), the observation matrix used in the rank experiments is
\[
  {\cal O}_L
  =
  \begin{bmatrix}
  R\\
  RA_\lambda\\
  \vdots\\
  RA_\lambda^L
  \end{bmatrix}
  \in
  \mathbb R^{(L+1)P\times Pr}.
\]
Thus \({\cal O}_Lx\) is the parent-average history generated by the fine
state \(x\).

All reported matrix ranks are computed in double precision.  The effective
rank is the SVD diagnostic defined in
\Cref{sec:stable-observability}.  It is tolerance-dependent and should not
be confused with the exact algebraic rank.  In the time-dependent
finite-volume experiments, the horizontal axis is the time-step index \(n\).
Since
\[
  \Delta t=\frac{\lambda h}{a},
\]
equal numbers of time steps do not represent equal physical times for
different \(\lambda\).  When \(h=1\) and \(a=1\), the physical time is
\(t_n=n\lambda\), and the number of full-domain advective periods is
\(n\lambda/N\).

\subsection{Experiment 1: exact rank versus effective rank}

Take \(P=8\), \(r=6\), and
\[
  \lambda\in\{1,0.5,0.1,0.01,0.001\}.
\]
For \(L=0,\ldots,r+3\), form
\[
  {\cal O}_L =
  \begin{bmatrix}
  R\\
  RA_\lambda\\
  \vdots\\
  RA_\lambda^L
  \end{bmatrix}.
\]
The exact theorem predicts the rank sequence
\[
  8,\;15,\;22,\;29,\;36,\;43,\;43,\ldots .
\]
The effective numerical rank is computed using the definition in
\Cref{sec:stable-observability}, namely
\[
  \rank_\tau({\cal O}_L)
  =
  \card\{i:s_i({\cal O}_L)>\tau\},
  \qquad
  \tau
  =
  100\epsmach
  \max\{(L+1)P,Pr\}
  \|{\cal O}_L\|_2 .
\]

\begin{figure}[htbp]
\centering
\includegraphics[width=0.75\textwidth]{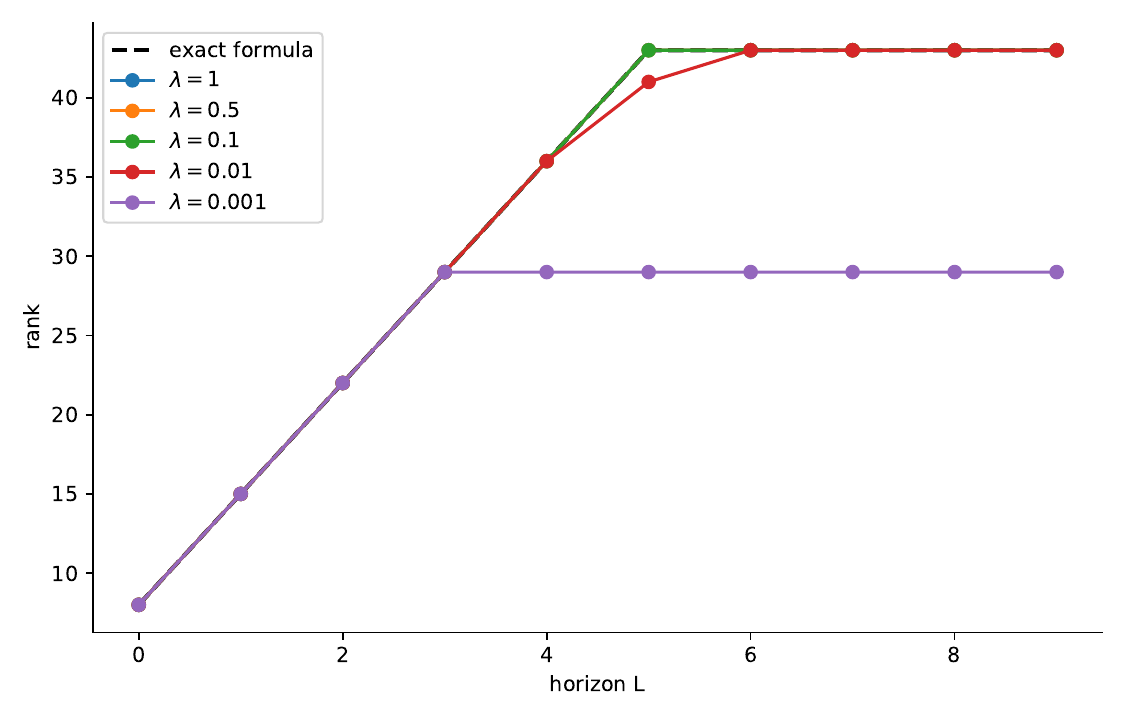}
\caption{Exact rank and effective floating-point rank of the observation
matrix \({\cal O}_L\) for \(P=8\), \(r=6\), and \(L=0,\ldots,9\).
The black dashed curve is the exact formula
\(P+(P-1)\min\{L,r-1\}\), giving the sequence
\(8,15,22,29,36,43,43,\ldots\).  The colored curves show the effective
numerical rank computed from the singular values using the tolerance
\(\tau=100\epsilon_{\rm mach}\max\{(L+1)P,Pr\}\|{\cal O}_L\|_2\).
Several colored curves are not separately visible because their effective
ranks coincide with one another, and in some cases also coincide with the
black dashed exact-rank curve.  This overlap is expected and indicates that
the declared SVD tolerance detects the full algebraic rank at those horizons.
Visible separation from the dashed curve indicates effective-rank loss.}
\label{fig:rank-effective}
\end{figure}

The diagnostic point of \Cref{fig:rank-effective} is not that the theorem
fails.  The exact rank is independent of floating-point tolerance for every
\(\lambda>0\).  The overlapping curves should be read as agreement, not as
missing data: for those horizons, the effective rank is identical for
several Courant numbers and often coincides with the exact-rank formula.
The visible departures from the dashed curve show where exact rank and
stable observable rank differ.  Small \(\lambda\) makes the deepest delayed
directions nearly invisible under the declared SVD tolerance.

One feature of \Cref{fig:rank-effective} is that the effective rank may
increase after the exact rank has already saturated.  This is possible
because later rows of \({\cal O}_L\) do not add new algebraic directions, but
they can strengthen already-visible directions.  Equivalently, the
finite-horizon observability Gramian
\[
  W_L
  =
  {\cal O}_L^\top{\cal O}_L
  =
  \sum_{t=0}^{L}
  (A_\lambda^t)^\top R^\top R A_\lambda^t
\]
satisfies
\[
  W_{L+1}
  =
  W_L
  +
  (A_\lambda^{L+1})^\top R^\top R A_\lambda^{L+1}
  \succeq W_L.
\]
Thus no observed direction becomes weaker when another observation time is
added, although the algebraic row space may already be saturated.  Longer
observation can therefore improve stable recoverability without changing
the exact rank.

\subsection{Experiment 2: saturated singular values}

The rank plot in \Cref{fig:rank-effective} compresses singular-value
information into a single integer.  To see where the effective-rank loss
comes from, we fix the horizon at the saturated value \(L=r-1=5\) and plot
the singular values of the single matrix \({\cal O}_{r-1}\)
(\Cref{fig:singular-values-saturated}).  Thus the horizontal axis in \Cref{fig:singular-values-saturated} is not time or horizon; it is the singular-value index.  For the present parameters,
\[
  {\cal O}_{r-1}\in\mathbb R^{48\times48},
  \qquad
  \rank{\cal O}_{r-1}=43,
\]
so exactly five singular values vanish in exact arithmetic.  The effective rank is the number of singular values that lie above the declared SVD tolerance.

\begin{figure}[htbp]
\centering
\includegraphics[width=0.78\textwidth]{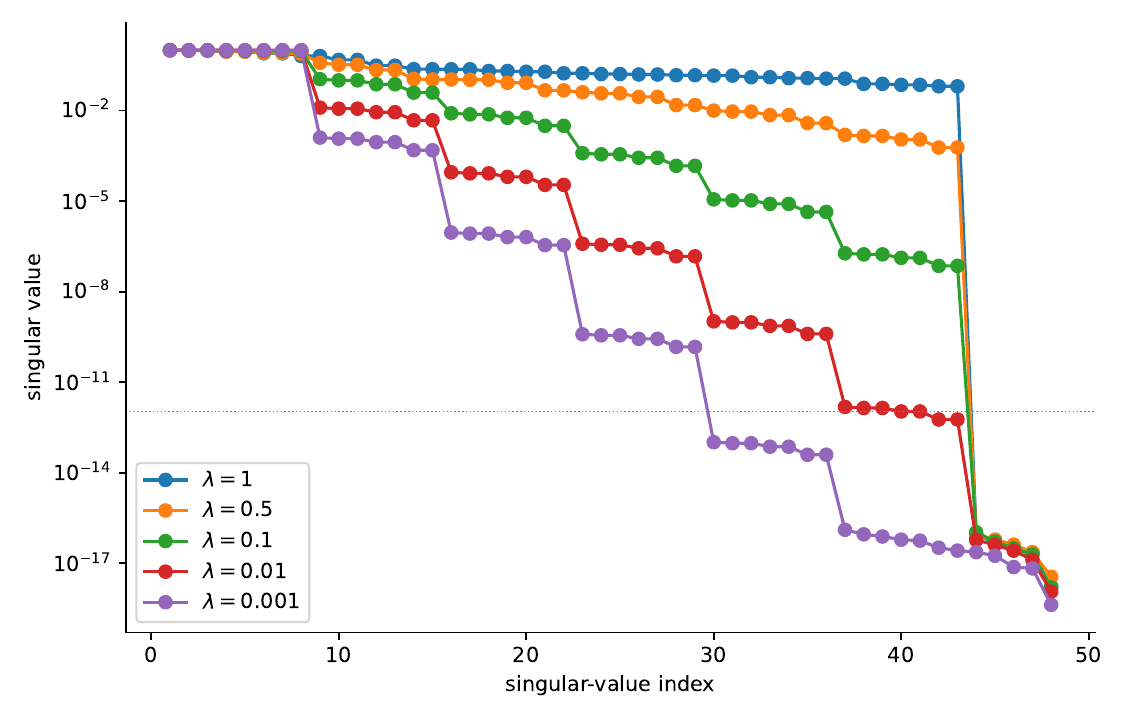}
\caption{
Singular values of the saturated observation matrix \({\cal O}_{r-1}\) for \(P=8\), \(r=6\), and \(L=r-1=5\). Unlike \Cref{fig:rank-effective}, this figure does not plot rank versus horizon.  The horizon is fixed, and the horizontal axis is the singular-value index \(i=1,\ldots,48\).  Each curve shows the singular values \(s_i({\cal O}_{r-1})\) in decreasing order for one value of \(\lambda\).  The exact saturated rank is \(43\), so five singular values are zero in exact arithmetic.  The effective rank reported in \Cref{tab:effective-rank-saturation} is obtained by counting how many singular values lie above the declared tolerance.  For small \(\lambda\),
some algebraically nonzero singular values fall below this tolerance, which explains the effective-rank loss seen in \Cref{fig:rank-effective}. The dotted horizontal line marks the corresponding SVD tolerance; for these
parameters the tolerance levels nearly coincide across \(\lambda\).
}
\label{fig:singular-values-saturated}
\end{figure}

\begin{table}[htbp]
\centering
\caption{
Effective rank at the saturated horizon \(L=r-1=5\) for \(P=8\), \(r=6\).
The exact rank is \(43\).  The effective rank uses the same SVD tolerance as
\Cref{fig:rank-effective}.  The following table quantifies how small \(\lambda\) can
make algebraically visible directions numerically invisible.
}
\begin{tabular}{@{}cccc@{}}
\toprule
\(\lambda\) & Exact rank & Effective rank & Tolerance \(\tau\) \\
\midrule
\(1\)     & \(43\) & \(43\) & \(1.066\times10^{-12}\) \\
\(0.5\)   & \(43\) & \(43\) & \(1.066\times10^{-12}\) \\
\(0.1\)   & \(43\) & \(43\) & \(1.066\times10^{-12}\) \\
\(0.01\)  & \(43\) & \(41\) & \(1.066\times10^{-12}\) \\
\(0.001\) & \(43\) & \(29\) & \(1.066\times10^{-12}\) \\
\bottomrule
\end{tabular}
\label{tab:effective-rank-saturation}
\end{table}

\Cref{tab:effective-rank-saturation} gives the numerical consequence of
the singular-value threshold at the saturated horizon.  For
\(\lambda=1\), \(0.5\), and \(0.1\), the effective rank equals the exact
rank \(43\).  For \(\lambda=0.01\), two algebraically nonzero directions
fall below tolerance, and for \(\lambda=0.001\) the effective rank drops to \(29\).  This is not a contradiction of \Cref{thm:rank}.  It shows that the
exactly visible delayed layers can be numerically unrecoverable when their
singular values are near the floating-point threshold.

The multiplier \(C_{\rm svd}=100\) in the threshold is a declared numerical
diagnostic, not a mathematical constant.  Changing it can change the
reported effective rank when singular values lie near the threshold.  For
this reason, \Cref{fig:singular-values-saturated} is included alongside the
rank table: it exposes the singular spectrum directly, so that other
reasonable thresholds can be assessed from the same data.

\subsection{Experiment 3: queue conditioning}

For \(q=1,\ldots,r-1\), construct the triangular queue map \(T_q\) from
\Cref{thm:conditioning}.  Compute
\[
  \sigma_{\min}(T_q)
  \qquad
  \text{and}
  \qquad
  \lambda^q
\]
over a logarithmic grid of \(\lambda\) values.  The figure displays the
smallest singular value and the reference scaling \(\lambda^q\).  The
condition-number behavior is given analytically in \Cref{thm:conditioning}.

\begin{figure}[htbp]
\centering
\includegraphics[width=0.75\textwidth]{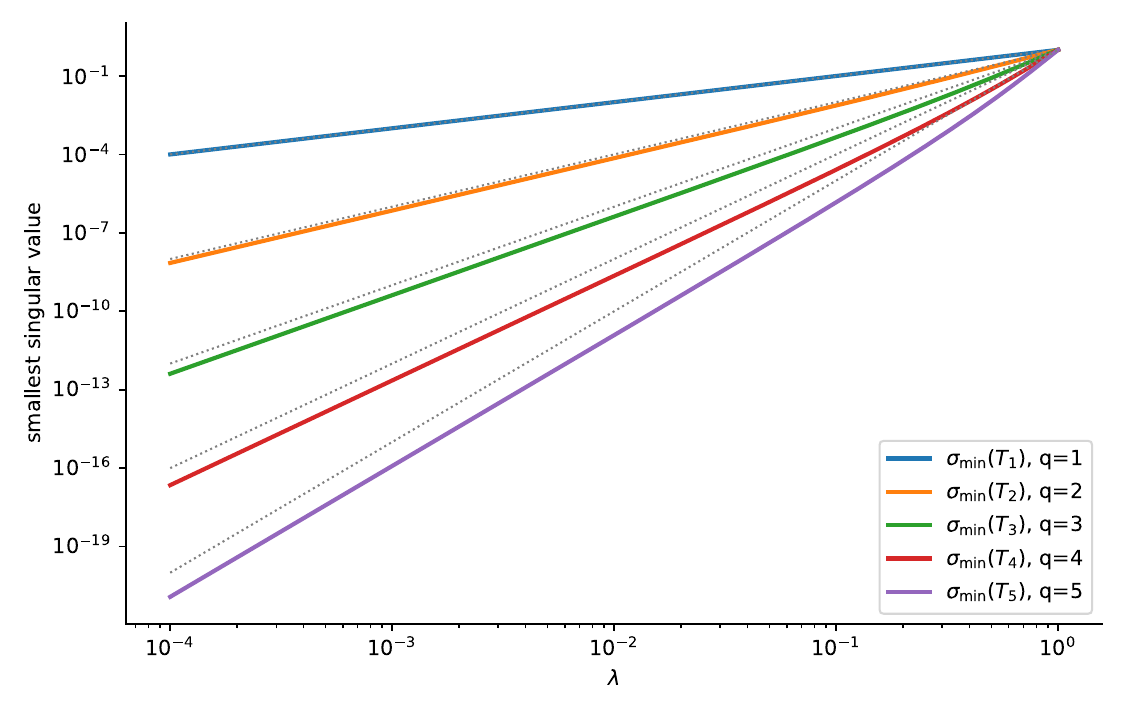}
\caption{Conditioning of the collar-to-queue map \(T_q\) from
Theorem~\ref{thm:conditioning}.  Each solid curve shows
\(\sigma_{\min}(T_q)\) as a function of \(\lambda\) for
\(q=1,\ldots,5\).  The dotted reference curves show \(\lambda^q\).
The deepest exposed collar coordinate enters the queue with a factor of
order \(\lambda^q\), so exact recovery becomes ill-conditioned when
\(\lambda\) is small.  This explains why exact observability and
floating-point effective observability can differ.}
\label{fig:queue-conditioning}
\end{figure}

\Cref{fig:queue-conditioning} gives a local explanation of the global
SVD behavior.  The curves follow the powers of \(\lambda\) predicted by
\Cref{thm:conditioning}: the deepest exposed collar coordinate enters the
queue with a coefficient of order \(\lambda^q\).  Thus the observation matrix
can have full exact rank while the information needed to recover a deep
collar layer is amplified by a large inverse factor.  The figure supports
the distinction between exact observability and stable finite-precision
observability. For very small \(\lambda\), some displayed singular values lie below the level at which double precision can be expected to provide high relative accuracy for the smallest singular value.  Those points are included only as a visual diagnostic of the trend, not as high-relative-accuracy numerical measurements.  The conclusion does not rely on those tiny computed values: the scaling and conditioning are established analytically in \Cref{thm:conditioning}.

\subsection{Experiment 4: periodic step-function flattening}

Use \(N=64\) fine cells and the periodic step
\[
  u_i^0 =
  \begin{cases}
  1, & 0\le i<N/2,\\
  0, & N/2\le i<N.
  \end{cases}
\]
In this experiment, we write the computed fine-grid solution as \(u^n\).

Run the scalar upwind update for \(\lambda=0.5\), \(\lambda=0.1\), and the
control case \(\lambda=1\).  The comparison is made per update step \(n\).  Since \(\Delta t=\lambda h/a\), the three curves do not correspond to equal physical times at the same value of \(n\).  The purpose of this experiment is to show per-update contraction and finite-grid flattening.  If one wants to compare at equal advective time, the relevant nondimensional time is \(n\lambda/N\). Record
\[
  E_2(n)=\|u^n-\bar u^n\mathbf 1\|_2,
\]
\[
  E_\infty(n)=\|u^n-\bar u^n\mathbf 1\|_\infty,
\]
\[
  \TV(n)=\sum_i |u_i^n-u_{i-1}^n|,
\]
and
\[
  M(n)=|\bar u^n-\bar u^0|.
\]

\begin{figure}[htbp]
\centering
\includegraphics[width=0.85\textwidth]{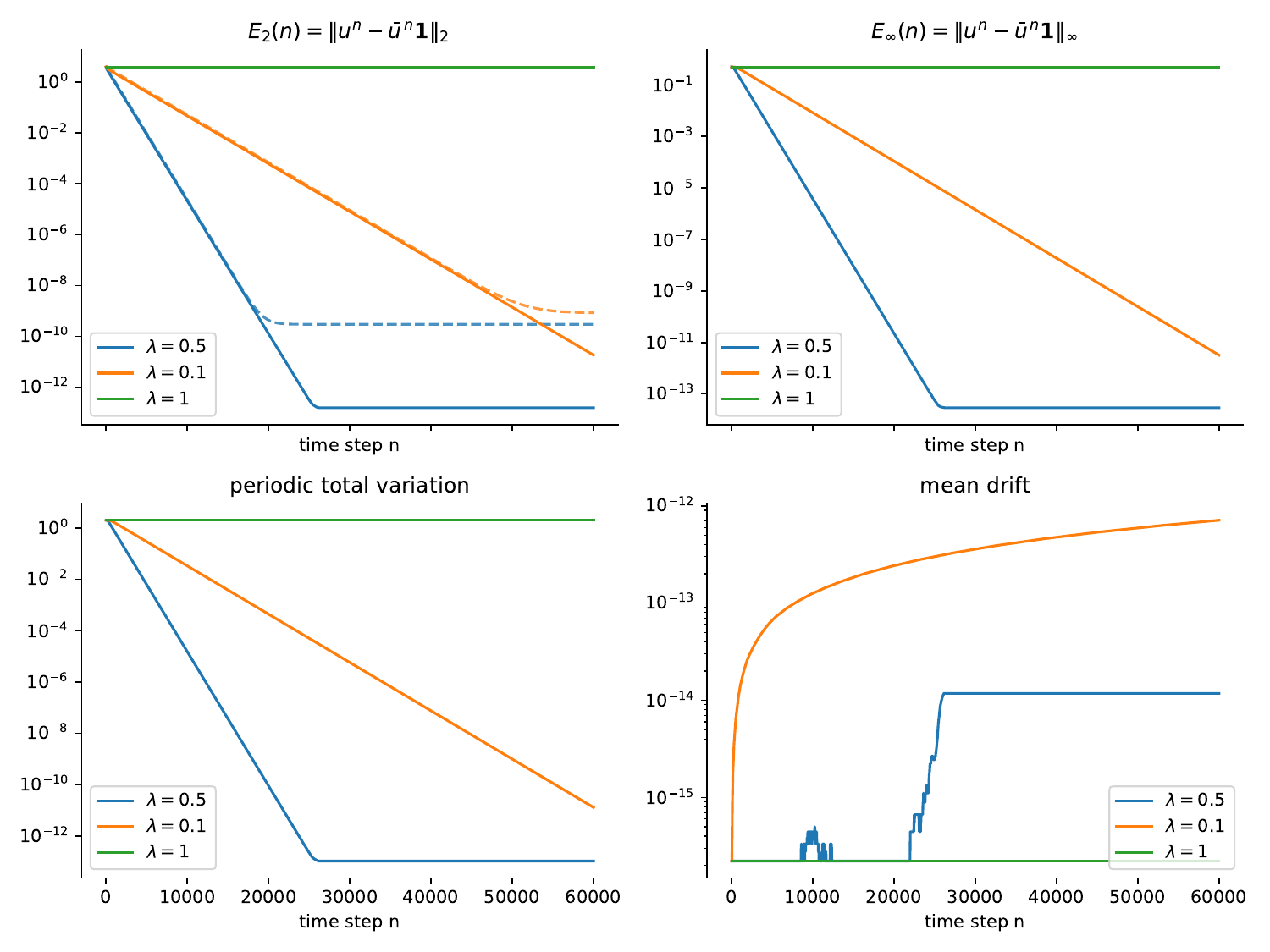}
\caption{Periodic step-function flattening for \(N=64\) fine cells and
\(\lambda\in\{0.5,0.1,1\}\).  The initial condition is \(1\) on the first
half of the periodic domain and \(0\) on the second half, so the conserved
mean is \(1/2\).  Top left: the \(L^2\) deviation
\(E_2(n)=\|u^n-\bar u^n\mathbf 1\|_2\); dashed curves show the theoretical
bounded-perturbation envelope for \(0<\lambda<1\), using
\(\eta=200\epsilon_{\rm mach}\sqrt N\).  Top right: the maximum-norm
deviation \(E_\infty(n)=\|u^n-\bar u^n\mathbf 1\|_\infty\).  Bottom left:
the periodic total variation
\(\mathrm{TV}(n)=\sum_i |u_i^n-u_{i-1}^n|\).  Bottom right: the mean drift
\(M(n)=|\bar u^n-\bar u^0|\), plotted on a logarithmic scale with a small
positive floor for display.  For \(0<\lambda<1\), the nonconstant component
decays to the declared perturbation floor.  For \(\lambda=1\), the method is a
cyclic shift and the step does not flatten.}
\label{fig:step-flattening}
\end{figure}

The results are presented in \Cref{fig:step-flattening}.  This is the
simplest finite-volume illustration of dissipative decay.  The step initially
contains sharp unresolved structure.  First-order upwind smears that
structure for \(0<\lambda<1\), and the nonconstant component decays about
the computed mean until it reaches the perturbation floor.  The dashed
curves in the \(E_2\) panel use
\[
  \eta=200\epsilon_{\rm mach}\sqrt N.
\]
This value is a conservative diagnostic perturbation floor used to visualize
the bounded-perturbation estimate.  It is not derived as a sharp
implementation-specific IEEE 754 floating-point rounding constant.  The \(\lambda=1\) run is important
because it separates advection from dissipation: in that case the method is
a cyclic shift and no flattening occurs.

\begin{table}[htbp]
\centering
\caption{
Final diagnostics for the periodic step-function experiment at \(n=60{,}000\). The dissipative cases \(0<\lambda<1\) approach a finite-precision floor, while the control case \(\lambda=1\) preserves the nonconstant profile by cyclic transport.
}
\begin{tabular}{@{}ccccc@{}}
\toprule
$\lambda$ & Final $E_2$ & Final $E_\infty$ & Final TV & Max mean drift \\
\midrule
1    & $4.000\times10^{0}$   & $5.000\times10^{-1}$  & $2.000\times10^{0}$   & $0$ \\
0.5  & $1.537\times10^{-13}$ & $2.909\times10^{-14}$ & $1.048\times10^{-13}$ & $1.177\times10^{-14}$ \\
0.1  & $1.815\times10^{-11}$ & $3.206\times10^{-12}$ & $1.282\times10^{-11}$ & $7.110\times10^{-13}$ \\
\bottomrule
\end{tabular}
\label{tab:step-final-diagnostics}
\end{table}

The values in Table~\ref{tab:step-final-diagnostics} quantify the visual
message of \Cref{fig:step-flattening}.  The nondissipative case
\(\lambda=1\) preserves the step by cyclic transport: the final \(E_2\),
\(E_\infty\), and total variation are unchanged from their initial values.
For \(\lambda=0.5\), the nonconstant component is reduced to roundoff-level
values by \(n=60{,}000\).  For \(\lambda=0.1\), the decay per update is
slower, as predicted by the larger value of \(\rho_N\), but the same
flattening mechanism is observed.  In all cases the measured mean drift
remains at roundoff level, so the observed flattening is about a computed
mean that remains very close to the initial mean in these runs.

\subsection{Experiment 5: delayed collision and erasure}

This experiment addresses the main closure question directly.  Take
\(P=8\), \(r=8\), \(N=64\), \(\lambda=0.5\), and construct two states \(x^0\)
and \(y^0\) with the same parent averages.  We use \(m=3\), \(\ell=r-1-m=4\), and perturbation amplitude
\(\delta_0=0.25\).  The reference state is \(y_i^0\equiv0.5\).  The
perturbed state \(x^0\) equals \(y^0\) except in parent \(K=0\), where
\[
  x_{0,0}^0 = 0.5-\delta_0,
  \qquad
  x_{0,\ell}^0 = 0.5+\delta_0.
\]
All other child values are \(0.5\).  Hence \(Rx^0=Ry^0\).  The compensated
perturbation remains inside the same parent through times \(0,\ldots,m\),
so the parent histories agree initially.  The first possible parent-level
separation occurs at time \(m+1=4\), when part of the perturbation has
reached the parent interface.

Measure
\[
  D(n)=\|RA_\lambda^n x^0-RA_\lambda^n y^0\|_2 .
\]

\begin{figure}[htbp]
\centering
\includegraphics[width=0.85\textwidth]{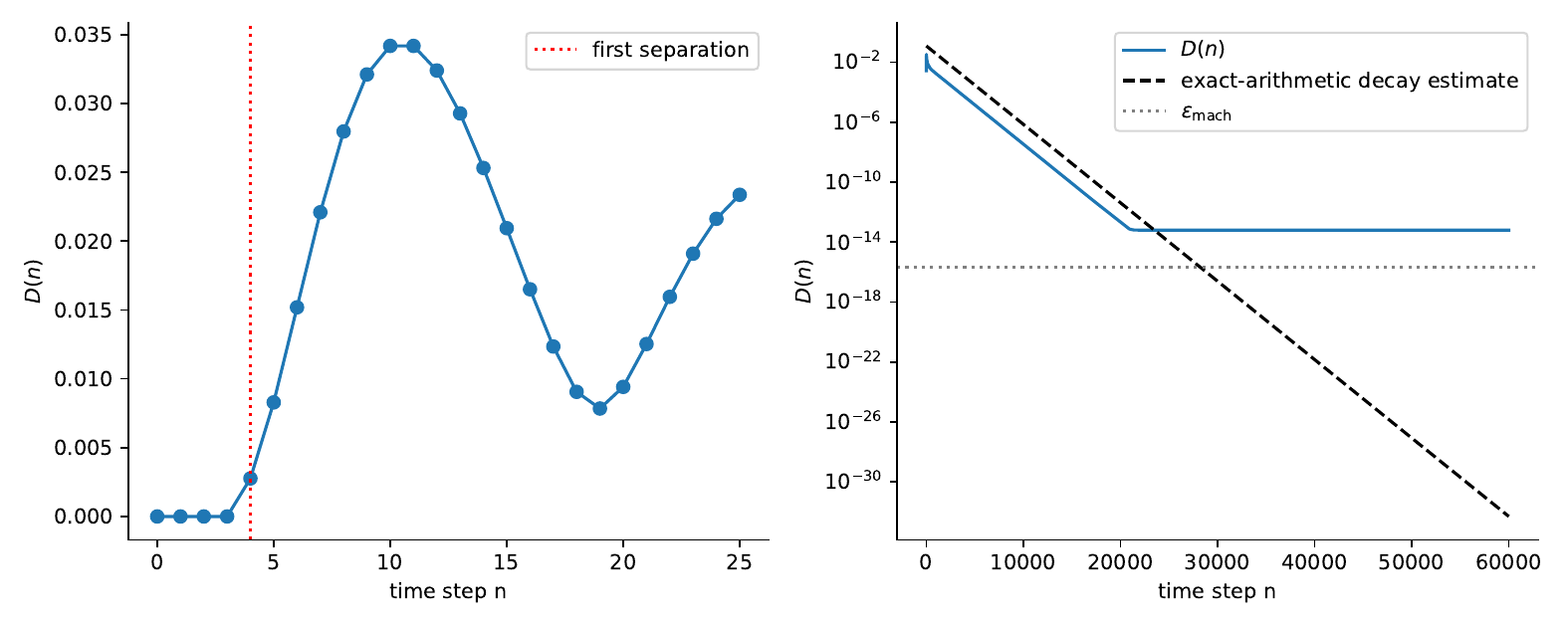}
\caption{Delayed collision and dissipative decay for \(P=8\), \(r=8\),
\(N=64\), and \(\lambda=0.5\).  The two initial fine states have identical
parent averages, \(Rx^0=Ry^0\), but differ inside one parent by a compensated
subcell perturbation.  The plotted quantity is
\(D(n)=\|RA_\lambda^n x^0-RA_\lambda^n y^0\|_2\), the difference between the
two parent-average states at time step \(n\).  Left: short-time behavior.
The red dotted line marks the first possible separation time \(n=m+1=4\).
Before this time, the hidden perturbation has not yet reached a parent
interface, so the parent averages agree.  After this time, the hidden
subcell information becomes visible at the coarse level.  Right: long-time
behavior on a logarithmic scale.  The blue curve decays because the upwind
scheme is dissipative.  The dashed curve is the exact-arithmetic decay
estimate \(r^{-1/2}\rho_N^n\|x^0-y^0\|_2\). Once this estimate falls below the accumulated roundoff scale, the computed
curve reaches an implementation-dependent numerical plateau.  The gray dotted
line marks \(\epsilon_{\rm mach}\) only as a unit-roundoff reference; it is
not expected to coincide with the observed plateau.}
\label{fig:delayed-collision}
\end{figure}

\Cref{fig:delayed-collision} combines the two main mechanisms in the paper. The left panel shows that mean-only closure is not exact: two fine states with identical parent averages can produce different future parent averages after a finite delay.  The right panel shows that exact memory and practical long-time memory are different.  Once the hidden perturbation becomes visible at the parent level, the dissipative upwind scheme contracts it.  In exact
arithmetic the decay estimate continues toward zero, while in double precision the computed difference reaches an implementation-dependent
roundoff plateau.  Thus the same perturbation is first hidden, then observable, and finally reduced to the numerical floor of the computation.

\subsection{Verification check: one-step flux closure}

A final sanity check uses random scalar fine data and verifies the identity
in \Cref{prop:flux-closure}.  For the scalar upwind scheme,
\[
  \Phi_{K}=\lambda x_{K,r-1}
\]
is the realized right-face flux of parent \(K\).  The parent update computed
by restricting the fine solution is compared with
\[
  y_K^{n+1}=y_K^n-\frac1r(\Phi_K-\Phi_{K-1}).
\]
The check uses \(P=10\), \(r=4\), \(N=40\), and \(\lambda=0.37\).  The fine
states are drawn independently from the standard normal distribution using
NumPy's default random-number generator with seed \(12345\).  For each of
100 trials, the discrepancy is measured in the maximum norm over parent averages:
\[
  \left\|
  R A_\lambda x
  -
  \left(
  Rx-\frac1r(\Phi-\operatorname{roll}(\Phi,1))
  \right)
  \right\|_\infty .
\]
The maximum discrepancy over the 100 trials was \(4.44\times10^{-16}\) (twice \(\epsmach\)), well below \(10^{-12}\) and consistent with roundoff. This check is included only as software validation of the telescoping
identity in \Cref{prop:flux-closure}.  The substantive numerical evidence is provided by the rank, conditioning, step-flattening, and delayed-collision
experiments.

Taken together, the experiments support the three main claims of the paper. The rank and singular-value experiments compare the numerical construction against the exact memory formula and show where effective rank can differ from exact rank.  The queue-conditioning experiment identifies the local temporal mechanism responsible for loss of stable observability, while \Cref{prop:spatial-conditioning} records the separate spatial conditioning introduced by the periodic difference operator.  The step-function and delayed-collision experiments then show the finite-volume consequences: unresolved information may be needed for exact short-time pathwise closure, but for \(0<\lambda<1\) the same information can be dissipated below a declared numerical tolerance.

\section{Discussion} \label{sec:discussion}

The results clarify a distinction that is easy to miss in coarse CFD and
super-resolution settings.  Conservation at one time is not the same as
pathwise predictive closure.  A reconstruction may exactly preserve parent
averages, yet place material differently inside a parent.  In an upwind
finite-volume method, those placements determine when unresolved data reach
parent interfaces.  Exact future parent means can therefore require memory
beyond the current parent means.

The numerical results in \Cref{sec:num_exp} are meant to be read in the same order as
the analysis.  \Cref{fig:rank-effective} and \Cref{tab:effective-rank-saturation} show the gap between exact rank and tolerance-dependent effective rank.  \Cref{fig:singular-values-saturated}
shows this gap at the singular-value level, while
\Cref{fig:queue-conditioning} identifies the triangular delayed-queue
map as the local source of poor conditioning.  \Cref{fig:step-flattening}
and \Cref{tab:step-final-diagnostics} show the dissipative decay predicted by \Cref{thm:erasure}.  Finally, \Cref{fig:delayed-collision} combines the two themes: hidden subcell information can affect future parent means after a delay, yet its effect can subsequently decay because of numerical dissipation.

The rank formula gives the exact algebraic requirement for the scalar
periodic upwind hierarchy.  Each newly exposed child layer contributes
\(P-1\) independent periodic differences.  The missing one dimension at each level is the constant periodic flux gauge: adding the same face-flux value everywhere does not change any parent average.  At saturation, a centralized encoder saves exactly \(r-1\) coordinates compared with raw local queues.

The conditioning theorem gives a different message.  Although every positive \(\lambda\) exposes the same algebraic directions, the deepest exposed layer can appear with coefficient \(\lambda^q\).  For small \(\lambda\), a direction may be visible in exact arithmetic and useless in double precision.  This is why the effective rank experiment is necessary.  A numerical paper should not report exact dimension without discussing conditioning.

The bounded-perturbation decay theorem gives the final piece.  For
\(0<\lambda<1\), the upwind method is diffusive.  Nonconstant fine-grid
information contracts at rate \(\rho_N\), up to the declared perturbation floor.  Thus exact short-time memory and long-time practical memory can differ.  The step-function experiment shows this visibly: the sharp initial discontinuity is flattened about the computed mean, whose drift is measured to remain at roundoff level in the reported runs.  The delayed-collision experiment combines both effects: two fine states with equal parent averages eventually separate at the coarse level, then their difference is dissipated.

These results do not contradict the value of learned closures, super-resolution, or high-order CFD.  They say that a meaningful claim about
coarse prediction must specify the information available to the model, the
prediction horizon, the tolerance, the numerical method, and the error
criterion.  Recent CFD and AI studies use different information sets and
different goals, including reconstruction quality, short-time forecasting, statistical fidelity, entropy stability, or engineering prediction
\cite{Sofos:2025spatiotemporal,
Liu:2026entropyCFN,Charles:2025dspweno}.  The present theorem concerns exact
pathwise parent-mean prediction in one scalar finite-volume hierarchy.

\subsection{Limitations and future work}
\label{sec:limitations}

The analysis is intentionally narrow.  The grid is uniform and periodic, the
dynamics are linear, the time integrator is forward Euler, and the numerical
flux is first-order upwind.  The encoder lower bounds are topological
dimension statements for continuous real-valued encoders on open sets.  They
are not bit-complexity bounds, statistical learning bounds, or claims about
arbitrary discontinuous encoders.  The results also do not establish the same
rank formula for nonlinear conservation laws, compressible Euler dynamics,
nonuniform or unstructured grids, nonperiodic boundaries, high-order
reconstruction, Runge--Kutta stages, limiters, entropy constraints, stochastic
closures, or learned models.

These limitations are useful because they identify the next problems.  A
natural extension is to repeat the same program for higher-order scalar
schemes: identify the exact observable subspace, quantify its conditioning,
and test which parts are preserved or erased by numerical dissipation.
Another direction is to study nonperiodic boundaries, where the periodic flux
gauge is replaced by boundary-dependent information.  A carefully scoped
Euler study should begin with the one-step flux identity of
\Cref{prop:flux-closure}, then ask whether a finite predictive register can
be defined for a declared solver, limiter, and time integrator.  More complex
finite-volume closures may involve flux histories, interface states,
residual stresses, or memory terms
\cite{ChorinHaldKupferman:2002memory,ParishDuraisamy:2017nonmarkovian,
Agdestein:2026exactfvclosure}.  The same three questions should guide such
extensions: what is the exact predictive state, how well conditioned is it,
and which parts of it survive numerical dissipation?

\section{Conclusions} \label{sec:conc}

Coarse finite-volume averages are generally not exact predictive states for
the coarse evolution of a realized fine solution.  In the periodic scalar
upwind hierarchy studied here, the exact finite-horizon memory requirement
is
\[
  (P-1)\min\{L,r-1\}
\]
real coordinates beyond the current parent averages.  A flux-divergence
queue attains this bound, and periodicity removes one common flux-gauge
coordinate per queue time.

The main point is not only the rank formula.  Exact rank, stable recovery, and long-time numerical relevance are distinct.
Delayed collar information enters the queue with powers of \(\lambda\), so
algebraically visible modes can fall below an SVD tolerance when \(\lambda\)
is small.  For \(0<\lambda<1\), the same upwind method dissipates nonconstant modes, so unresolved subcell information that matters for exact short-time closure may later become irrelevant below a declared perturbation floor.

The numerical experiments separate these effects. The rank and singular value tests reproduce the predicted rank ladder and display the effective-rank loss. The queue-conditioning test explains the loss through the triangular
collar-to-queue map.  The step-function test shows flattening about the computed mean for \(0<\lambda<1\), while the \(\lambda=1\) control confirms that advection without numerical diffusion does not erase the step.  The delayed-collision test shows the closure mechanism directly: same-parent-mean states can separate at the coarse level after hidden subcell data reaches an interface, and the resulting difference is then damped.

The practical lesson is that a coarse model should not be judged only by
whether it preserves parent averages at one time.  One must also specify the
hidden information available to the model, the prediction horizon, the
floating-point tolerance, the numerical method, and whether the method
preserves or erases unresolved information over that horizon.  This paper establishes that perspective in a scalar periodic model; the limitations and extensions identified in \Cref{sec:limitations} describe the natural next steps.

\section*{Declarations}

\subsection*{Funding}
No external funding was received for this work.

\subsection*{Competing interests}
The authors declare that they have no financial or non-financial interests
that are directly or indirectly related to the work submitted for
publication.

\subsection*{Author contributions}
A.P. contributed to conceptualization, interpretation, verification, and writing.  N.C. contributed to conceptualization, mathematical analysis, numerical analysis, software, interpretation, verification, and writing.  D.D. contributed to conceptualization, numerical analysis, interpretation, verification, and writing.

\section*{Acknowledgments}

The authors used generative-AI tools in the course of this work, for
literature discovery, adversarial review of the mathematical
arguments, copyediting suggestions, and assistance with experiment
and verification code.  The research questions, results, and
conclusions are the authors' own.  The authors verified the cited
sources, mathematical arguments, numerical results, and final text,
and assume responsibility for all content.

\section*{Data and code availability}

The complete reproducibility package for this study is publicly
available in the Certified Simulation repository of the University of
Nicosia,
\url{https://github.com/UniversityOfNicosia/certified-simulation},
under \texttt{papers/finite-horizon-memory}.  The package contains
the Python scripts that generate every matrix test, scalar
finite-volume experiment, CSV file, LaTeX table fragment, and figure
reported in the manuscript, together with instructions for
reproducing the results, and executable certificates that replay the
paper's exact results in exact rational arithmetic.  No external CFD
solver or proprietary software is required.

\bibliography{sn-bibliography}

\end{document}